\documentclass[12pt]{amsart}

\usepackage{amsxtra,amssymb,amsthm,amsmath,amscd,url,listings}
\usepackage[utf8]{inputenc}
\usepackage{eucal}
\usepackage{fullpage}
\usepackage{scrtime}
\usepackage[colorlinks]{hyperref}

\newcounter{point}

\renewcommand{\leq}{\leqslant}
\renewcommand{\geq}{\geqslant}

\numberwithin{equation}{section}

\def\stacksum#1#2{{\stackrel{{\scriptstyle #1}}
{{\scriptstyle #2}}}}

\newcommand{\Cc}{\mathbf{C}}

\newcommand{\Aa}{\mathbf{A}}

\newcommand{\Zz}{\mathbf{Z}}
\newcommand{\Pp}{\mathbf{P}}

\newcommand{\Gg}{\mathbf{G}}

\newcommand{\Qq}{\mathbf{Q}}
\newcommand{\Fp}{{\mathbf{F}_p}}

\newcommand{\Ff}{\mathbf{F}}
\newcommand{\Fq}{\mathbf{F}_q}
\newcommand{\Fqt}{\mathbf{F}^\times_q}
\newcommand{\Fqn}{\Ff_{q^{\nu}}}
\newcommand{\bFq}{\bar{\Ff}_q}
\newcommand{\bQl}{\bar{\Qq}_{\ell}}
\newcommand{\bx}{\bar{x}}

\newcommand{\mods}[1]{\,(\mathrm{mod}\,{#1})}

\newcommand{\frtr}[3]{(\Tr{{#1}})({#2},{#3})}
\newcommand{\frfn}[1]{t_{{#1}}}

\newcommand{\ra}{\rightarrow}
\newcommand{\lra}{\longrightarrow}

\newcommand{\injecte}{\hookrightarrow}

\newcommand{\fleche}[1]{\stackrel{#1}{\lra}}

\DeclareMathOperator{\spec}{Spec}

\DeclareMathOperator{\rank}{rank}
\DeclareMathOperator{\pct}{pct}
\DeclareMathOperator{\kop}{T}

\DeclareMathOperator{\ord}{ord}

\DeclareMathOperator{\Imag}{im}

\DeclareMathOperator{\frob}{\mathrm{Fr}}

\DeclareMathOperator{\Tr}{tr}

\DeclareMathOperator{\swan}{Swan}
\DeclareMathOperator{\drop}{drop}

\DeclareMathOperator{\ft}{FT}
\DeclareMathOperator{\cond}{\mathbf{c}}

\newcommand{\eps}{\varepsilon}
\renewcommand{\rho}{\varrho}

\newcommand{\demi}{{\textstyle{\frac{1}{2}}}}

\newcommand{\sheaf}[1]{\mathcal{{#1}}}

\DeclareMathSymbol{\gena}{\mathord}{letters}{"3C}
\DeclareMathSymbol{\genb}{\mathord}{letters}{"3E}

\theoremstyle{plain}
\newtheorem{theorem}{Theorem}[section]
\newtheorem{lemma}[theorem]{Lemma}

\newtheorem{corollary}[theorem]{Corollary}

\newtheorem{proposition}[theorem]{Proposition}
\newtheorem*{proposition*}{Proposition}

\theoremstyle{remark}

\theoremstyle{definition}

\newtheorem{definition}[theorem]{Definition}

\newtheorem{example}[theorem]{Example}
\newtheorem{remark}[theorem]{Remark}

\newcommand{\mcL}{\mathcal{L}}

\newcommand{\mcF}{\mathcal{F}}
\newcommand{\mcG}{\mathcal{G}}
\newcommand{\mcH}{\mathcal{H}}
\newcommand{\mcK}{\mathcal{K}}

\renewcommand{\geq}{\geqslant}
\renewcommand{\leq}{\leqslant}

\begin{document}

\title{On the conductor of cohomological transforms}
 
\author{\'Etienne Fouvry}
\address{  Université Paris--Saclay,   CNRS \\
Laboratoire de Math\'ematiques d'Orsay\\
  91405 Orsay  \\France}
\email{etienne.fouvry@universite-paris-saclay.fr}
 \author{Emmanuel Kowalski}
\address{ETH Z\"urich -- D-MATH\\
  R\"amistrasse 101\\
  CH-8092 Z\"urich\\
  Switzerland} \email{kowalski@math.ethz.ch} \author{Philippe Michel}
\address{EPFL/SB/IMB/TAN, Station 8, CH-1015 Lausanne, Switzerland }
\email{philippe.michel@epfl.ch}

\date{\today,\ \thistime} 

\thanks{Ph. M. was partially supported by the SNF (grant
  200021-137488) and the ERC (Advanced Research Grant
  228304). \'E.F. thanks ETH Z\"urich, EPF Lausanne and the Institut
  Universitaire de France for financial support.  Ph.M. and E.K. were
  supported by the SNF-DFG Grant 200020L\_175755.}

\subjclass[2010]{11T23,14F20,11G20} 

\keywords{\'Etale cohomology, conductor, $\ell$-adic sheaves, Riemann
  Hypothesis over finite fields, exponential sums}

\begin{abstract}
  In the analytic study of trace functions of $\ell$-adic sheaves over
  finite fields, a crucial issue is to control the conductor of
  sheaves constructed in various ways.  We consider cohomological
  transforms on the affine line over a finite field which have trace
  functions given by linear operators with an additive character of a
  rational function in two variables as a kernel. We prove that the
  conductor of such transforms is bounded in terms of the complexity
  of the input sheaf and of the rational function defining the kernel,
  and discuss applications of this result, including motivating
  examples arising from the \textsc{Polymath8} project.
  
  \vspace{0.2cm}
\medskip 

{\noindent \textsc{Résumé.} } Dans l'étude analytique des fonctions
traces de faisceaux $\ell$--adiques sur les corps finis, un problème
crucial est de contrôler la taille du conducteur de faisceaux
construits de façons variées.  Nous considérons les transformées
cohomologiques sur la droite affine sur un corps fini dont les
fonctions traces sont données par des opérateurs linéaires dont la
matrice est un caractère additif évalué sur une fonction rationnelle
en deux variables. Nous prouvons que le conducteur de telles
transformées est borné en fonction du conducteur du faisceau de départ
et de la fraction rationnelle définissant le noyau. Enfin nous
présentons des applications et des exemples, en particulier des
exemples provenant du projet \textsc{Polymath}.

\medskip
\end{abstract}

\maketitle

{\small{
\setcounter{tocdepth}{1}
\tableofcontents}
}

\section{Introduction}

This paper considers a problem which appeared in special cases
in~\cite{FKM1,FKM2, FKM3} in our study of analytic applications of trace
functions over finite fields. We are given a constructible $\ell$-adic
sheaf $\sheaf{K}$ on $\Aa^1\times \Aa^1$ (or, potentially, on another
algebraic surface) over a finite field $\Fq$, and we use it to define
a ``cohomological transform'' with ``kernel'' $\sheaf{K}$, that maps a
constructible $\ell$-adic sheaf $\sheaf{F}$ on $\Aa^1_{\Fq}$ to
$$
T_{\sheaf{K}}(\sheaf{F})=R^1p_{1,!}(p_2^*\sheaf{F}\otimes\sheaf{K})(1/2),
$$
where $p_1$, $p_2$ are the two projections $p_i\,:\,
\Aa^1\times\Aa^1\ra \Aa^1$. The problem is then to \emph{estimate the
  conductor of $T_{\sheaf{K}}(\sheaf{F})$}, as defined in~\cite{FKM1},
in terms of that of $\sheaf{F}$.
\par
The arithmetic interpretation of this problem, and our motivation, is
that for suitable input sheaves $\sheaf{F}$ (as described later in
more detail), the trace function $\frfn{T_{\sheaf{K}}(\sheaf{F})}$ of
 is related to the trace functions
$\frfn{\sheaf{K}}$ and $\frfn{\sheaf{F}}$ by
$$
\frfn{T_{\sheaf{K}}(\sheaf{F})}(x)=-\frac{1}{\sqrt{q}}\sum_{y\in
  \Fq}\frfn{\sheaf{F}}(y)\frfn{\sheaf{K}}(x,y), 
$$
for all $x\in \Fq$. In other words, for all $\sheaf{F}$, we have
$\frfn{T_{\sheaf{K}}(\sheaf{F})}=\kop_K(\frfn{\sheaf{F}})$, where
$$
K(x,y)=\frfn{\sheaf{K}}(x,y)
$$
and $\kop_K$ is the (normalized) linear map defined on the space
$C(\Fq)$ of complex-valued functions on $\Fq$ by the kernel $K$,
i.e.,
\begin{equation}\label{TKdef}
\kop_K(\varphi)(x)=
-\frac{1}{\sqrt{q}}\sum_{y\in\Fq}K(x,y)\varphi(y). 	
\end{equation}
\par
The most important example of such transforms arises for
$K(x,y)=\psi(xy)$, where $\psi$ is a non-trivial additive character,
which corresponds to $\sheaf{K}=\sheaf{L}_{\psi(XY)}$ (where $X$, $Y$
are the coordinates on $\Aa^1\times\Aa^1$): the corresponding linear
operator $\kop_K$ on trace functions is (minus the) normalized Fourier
transform (which we denote also $\ft_{\psi}$) on $C(\Fq)$, namely
\begin{equation}\label{FTdef}
\ft_{\psi}(\varphi)(x)=-\frac{1}{\sqrt{p}}\sum_{y\in
  \Fp}\varphi(y)\psi(xy). 
\end{equation}
\par
The sheaf-theoretic construction, in that case, is due to Deligne, and
it was extensively studied by Laumon~\cite{laumon}.
\par
This special case is crucial in~\cite{FKM1} (and the following
papers). In particular, it is essential for our applications that we
have an estimate for the conductor of the Fourier transform
$\sheaf{G}$ in terms only of the conductor of $\sheaf{F}$, which
follows from the estimate
\begin{equation}\label{eq-ft}
\cond(\sheaf{G})\leq 10\cond(\sheaf{F})^2,
\end{equation}
proved in~\cite[Prop. 8.2]{FKM1}. In order to establish this result,
which we view as a form of ``continuity'' of the sheaf-theoretic
Fourier transform, we used the deep theory of the local Fourier
transform of Laumon~\cite{laumon,katz-esde}.
% , and in any case, these local results involve very subtle results
% in algebraic geometry.
\par
The general case of these transforms is a natural approach to
estimates for two-variable character sums (and more complicated
algebraic sums) based on Deligne's work, and an estimate for the
conductor leads for instance easily to strong statements of ``control
of the diagonal'' (see Proposition~\ref{pr-principle} for a precise
statement.)
\par
It is not at all clear if a local theory like Laumon's applies to the
general transforms we consider (the same applies to the theory of
``affine cohomological transforms'' of Katz~\cite{katz-act}, which has
diophantine applications to stratification results for sums of trace
functions.)  Thus, our present goal is to prove estimates for more
general cohomological transforms. These will be weaker
than~(\ref{eq-ft}), but more accessible. We will be able to do so when
the kernel $\sheaf{K}$ is a rather general Artin-Schreier sheaf, or in
other words (in the case when $q=p$ is prime) when
$$
K(x,y)=e\Bigl(\frac{f(x,y)}{p}\Bigr)
$$
for some rational function $f\in\Fp(X,Y)$.
\par
The precise statement is given in Theorem~\ref{th-conductor} in the
next section.  In the case of the Fourier transform, this gives a form
of the important property~(\ref{eq-ft}) which is more accessible than
Laumon's theory. Section~\ref{sec-polymath} treats this case fully, in
order to motivate and clarify the algebraic tools used in the general
case. In Section~\ref{sec-examples}, we discuss some first
applications of these bounds; for instance, we show how the ideas lead
to an account of the character sums considered by Conrey and Iwaniec
in~\cite{CI}. This section can, to a large extent, be read
independently of the part of the paper where the main results are
proved.

\begin{remark}
  (1) In work in progress, W. Sawin has developed a much more general
  and powerful theory of complexity measures of $\ell$-adic sheaves on
  schemes, including all so-called $6$ operations and derived category
  objects. His work will subsume ours entirely, but also it will
  involve much more difficult algebraic geometry. We (finally) submit
  the present paper for publication as an illustration of fairly
  simple manipulations of the formalism of étale cohomology, in the
  hope that it will be helpful to readers with a more analytic
  background.
  \par
  (2) In recent work, I. Petrow and M. Young~\cite{PY} generalized the
  estimate of Conrey and Iwaniec to more general characters. They need
  to estimate slightly different sums than those in~\cite{CI}, and the
  first draft of their preprint refered to this paper for this
  purpose. W. Sawin has also observed that their sums (as those of
  Conrey and Iwaniec) and special cases of hypergeometric sums, and
  can be directly estimated by an simple appeal to Katz's
  book~\cite{katz-esde}.
  \par
  (3) A slightly different definition of the conductor suggested by W.
  Sawin leads to better estimates (e.g., a linear bound for the
  conductor of the Fourier transform instead of~(\ref{eq-ft})). Since
  our work is in any case very restricted (see Remark (1)), and the
  most important qualitative feature is not affected for applications,
  we have not incorporated all the changes required by this
  adjustment.
\end{remark}

\subsection*{Acknowledgments.} Part of the original motivation for
this paper in 2013/2104 arose in online discussions related to the
\textsc{Polymath8} project. 
\par
We thank the referee for a detailed reading of the paper which led to significant simplications in several parts of the argument, and pointed out a
mistake in one of our applications.

\subsection*{Notation.}

By ``sheaf'', or ``$\ell$-adic sheaf'', we will always mean
``constructible $\bQl$-sheaf'', where $\ell$ will be a prime number
different from the characteristic of the base field.
\par
For a power $q\not=1$ of a prime $p$ and any integer $w\in\Zz$, a
\emph{$q$-Weil number of weight $w$} is an algebraic number
$\alpha\in\Cc$ such that all Galois-conjugates $\beta$ of $\alpha$
satisfy $|\beta|=q^{w/2}$.
\par
An algebraic variety over a field $k$ is a finite type, separated,
reduced scheme over $k$.  If $X/k$ is an algebraic variety over a
field $k$, and $\bar{k}$ is an algebraic closure of $k$, we denote by
$\bar{X}$ or $X_{\bar{k}}$ the base change $X\times \bar{k}$.
\par
For an algebraic variety $X$ over~$\Fq$ and an $\ell$-adic
sheaf~$\mcF$ on~$X$, we denote
$$
\chi_c(U,\mcF)=\sum_{i=0}^{2\dim(X)} (-1)^i\dim H^i_c(U\times
\bFq,\mcF).
$$
\par
If~$X$ is the affine line, we will abbreviate
$\chi(\sheaf{F})=\chi_c(\Aa^1,\sheaf{F})$, and we write
\begin{equation}\label{h1def}
h^1(\sheaf{F})=\dim H^1_c(\Aa^1_{\bar{\Fq}},\sheaf{F}).	
\end{equation}
\par
If $X_k$ is an algebraic variety over a field $k$ and $x\in X(k)$, we
denote by $\bx$ a geometric point above $x$. If $k$ is algebraically
closed, we take $\bx=x$. If $\sheaf{F}$ is an \'etale sheaf on $X$,
then $\sheaf{F}_{\bx}$ denotes the stalk of $\sheaf{F}$ at $\bx$.
% If $X$ is either defined over $\Fq$ or over $\bFq$, by $x\in X$, we
% mean that $x\in X(\bFq)$.
%%TODO check if this is enough
\par
Whenever a prime $\ell$ is given, we assume \emph{fixed} an
isomorphism $\iota\,:\, \bQl\ra \Cc$, and we use it as an
implicit identification.
\par
For any $\ell$-adic sheaf $\sheaf{F}$ on an algebraic variety
$X_{\Fq}$, we write
$\frfn{\sheaf{F}}(x)$ for the value at $x$ of the trace function of
$\sheaf{F}$, i.e., we have
\begin{equation}\label{tracefctdef}
	\frfn{\sheaf{F}}(x)=\iota(\frtr{\sheaf{F}}{\Fq}{x}),
\end{equation}
the trace of the Frobenius of $\Fq$ acting on the stalk of
$\sheaf{F}$ at $x$.
\par
If $k/\Fq$ is a finite extension, we write
$$
\frfn{\sheaf{F}}(x,|k|)=\frfn{\sheaf{F}}(x,k)=\iota(\frtr{\sheaf{F}}{k}{x}).
$$

\section{Statement of the main result}

We first recall the definition of the conductor of a constructible
$\ell$-adic sheaf $\sheaf{F}$ on the affine line over a finite field
$\Fq$.  Indeed, since in this work it will be important to work with
general constructible sheaves, and not only the middle-extension
sheaves considered in our previous works, we need to adapt the
definition slightly.
\par
Let $\sheaf{F}$ be a constructible $\ell$-adic sheaf over
$\Aa^1_{\Fq}$. Let $U\subset \Pp^1$ be the maximal dense open subset
where $\sheaf{F}$ is lisse. Let $j\,:\, U\injecte \Pp^1$ be the
corresponding open immersion. Recall that there is a canonical
(adjunction) map
$$
\sheaf{F}\lra j_*j^*\sheaf{F},
$$
and that $\sheaf{F}$ is said to be a \emph{middle-extension sheaf} if
this is an isomorphism. In general, if we let
$$
\sheaf{F}_0=j_*j^*\sheaf{F},
$$
then one shows that $\sheaf{F}_0$ is a middle-extension sheaf on
$\Pp^1_{\Fq}$, which is isomorphic to $\sheaf{F}$ when restricted to
$U$. We define
$$
\cond(\sheaf{F})=\rank(\sheaf{F}_0)+\sum_x
\swan_x(\sheaf{F}_0)+n(\sheaf{F})+\pct(\sheaf{F}),
$$
where:
\par
-- $n(\sheaf{F})=|(\Pp^1-U)(\bFq)|$ is the number of singularities of
$\sheaf{F}$ in $\Pp^1(\bFq)$;
\par
-- the sum is over $\Pp^1(\bFq)$, with all but finitely many terms
vanishing;
\par
-- we define
$$
\pct(\sheaf{F})=\dim H^0_c(\Aa^1\times\bFq,\sheaf{F}).
$$

\begin{remark}
  (1) If $\sheaf{F}$ is a middle-extension sheaf on $\Aa^1_{\Fq}$, we
  have $\sheaf{F}=\sheaf{F}_0$ (on $\Aa^1$) and
$$
\cond(\sheaf{F})=\rank(\sheaf{F})+\sum_x
\swan_x(\sheaf{F})+n(\sheaf{F}),
$$
as in our previous works. 
\par
(2) Let $\sheaf{P}$ be the kernel of the map
$$
\sheaf{F}\lra j_*j^*\sheaf{F}.
$$
\par
Then $\sheaf{P}$ has finite support; if this support is
$S\subset \Aa^1(\bFq)$, then
$$
|S|\leq \pct(\sheaf{F})\leq \sum_{s\in S} \dim \sheaf{F}_{\bar{s}}
$$ 
(see~\cite[\S 4.4, 4.5]{katz-sommes} for a discussion). 
\par
(3) Note that $n(\sheaf{F})$ takes into account the fact that a
general constructible sheaf might have ``artificial'' singularities,
which are not singularities of the associated middle-extension
sheaf. These may also be seen as the contribution to the conductor of
the cokernel $\sheaf{F}$ of the map
$$
\sheaf{F}\lra j_*j^*\sheaf{F},
$$
which is also a sheaf with finite support.
\par
For instance, let $U=\Pp^1-\Aa^1(\Fp)$ over $\Fp$, and let
$$
j\,:\, U\lra \Pp^1
$$ 
be the open immersion. Consider 
$$
\sheaf{F}=j_!\bQl,
$$
the extension by zero to $\Pp^1$ of the trivial sheaf on $U$. Then
$\sheaf{F}_0$ is the trivial sheaf on $\Pp^1$, with
$n(\sheaf{F}_0)=0$, and $\cond(\sheaf{F}_0)=1$, while
$\cond(\sheaf{F})=1+n(\sheaf{F})=1+|\Aa^1(\Fp)|=p+1$ because of the
artificial singularities created at the points in $\Aa^1(\Fp)$. It is
necessary here to have a big conductor if we want some basic
qualitative features of the Riemann Hypothesis to hold.
\end{remark}

We note the following useful property:
\begin{equation}\label{eq-subadd}
\cond(\sheaf{F}_1\oplus\sheaf{F}_2)\leq
\cond(\sheaf{F}_1)+\cond(\sheaf{F}_2)
\end{equation}
for two constructible sheaves on $\Aa^1$ (more generally, if
$$
0\lra \sheaf{F}_1\lra \sheaf{F}_3\lra \sheaf{F}_2\lra 0
$$
is a short exact sequence of constructible sheaves on $\Aa^1$, then we
have
$$
\cond(\sheaf{F}_3)\leq \cond(\sheaf{F}_1)+\cond(\sheaf{F}_2)
$$
as one can check.)
%%Use snake lemma to check that if F_1 and F_2 are lisse at x then F_3
%%is lisse at x, so n(F_3)\leq n(F_1)+n(F_2)
%%For Swan conductors, use exactness of the functors M --> M(x)
\par
We also recall the definition of the \emph{drop} of a constructible sheaf
$\sheaf{F}$ on $\Aa^1_{\bFq}$ at a point $x\in\Aa^1(\bFq)$: we have
\begin{equation}\label{eq-drop}
  \drop_x(\sheaf{F})=\rank(\sheaf{F}_0)-\dim \sheaf{F}_x,
\end{equation}
where $\sheaf{F}_x$ is the stalk of $\sheaf{F}$ at $x$. Note that the
rank of $\sheaf{F}_0$ is also the ``generic'' rank of $\sheaf{F}$,
i.e., the dimension of the fiber at a geometric generic point.
\par
\medskip
\par
As mentioned in the introduction, we consider in this paper a kernel
$\sheaf{K}$ which is an Artin-Schreier sheaf, with trace function
given by additive characters of rational function. We give a formal
definition to avoid any ambiguity concerning the behavior at the poles
or points of indeterminacy of a rational function in two variables.

\begin{definition}[Artin-Schreier sheaf on $\Aa^n$]\label{def-as}
  Let $\Fq$ be a finite field of characteristic $p$, and let
  $\ell\not=p$ be a prime number, and $\psi$ a non-trivial additive
  $\ell$-adic character of $\Fq$. Let $\sheaf{L}_{\psi}$ denote the
  associated Artin-Schreier sheaf on $\Aa^1_{\Fq}$ (see~\cite[Sommes Trig.]{deligne} for the precise definition).
\par
Let $f\in\Fq(X_1,\ldots, X_n)$ be a rational function for some $n\geq
1$. Write $f=f_1/f_2$ where $f_i\in \Fq[X_1,\ldots,X_n]$ and where
$f_1$ is coprime with $f_2$. Let $U\subset \Aa^n$ be the open set
where $f_2$ is invertible, $j\,:\, U\injecte \Aa^n$ the corresponding
open immersion, and let
$$
f_U\,:\, U\lra \Aa^1
$$
be the morphism associated to the restriction of $f$ to $U$.
\par
The \emph{Artin-Schreier sheaf on $\Aa^n$ associated to $f$} is the
constructible $\ell$-adic sheaf on $\Aa^n_{\Fq}$ given by
$$
\sheaf{L}_{\psi(f)}=j_!f_U^*\sheaf{L}_{\psi}.
$$
\par
We also write $\sheaf{L}_{\psi(f(X_1,\ldots,X_n))}$ for this sheaf. We
define its conductor to be
$$
\cond(\sheaf{L}_{\psi(f)})=1+\deg(f_1)+\deg(f_2),
$$
and we will also sometimes just speak of the conductor $\cond(f)$ of
$f$.
\end{definition}

We will find a satisfactory generalization of~(\ref{eq-ft}) for
transforms associated to  a kernel which is an Artin-Schreier sheaf. 

\begin{theorem}[Conductor of Artin-Schreier
  transforms]\label{th-conductor}
  Let $\Fq$ be a finite field of order $q$ and characteristic $p$,
  $\ell$ a prime distinct from $p$. Let $\sheaf{K}$ be an $\ell$-adic
  sheaf on $\Aa^1\times \Aa^1$ over $\Fq$ of the form
$$
\sheaf{K}=\sheaf{L}_{\psi(f(X,Y))},
$$
where $\psi$ is a non-trivial additive $\ell$-adic character and
$f\in\Fq(X,Y)$ is a rational function with conductor $<p$. 
\par
For constructible sheaves $\sheaf{F}$ on $\Aa^1_{\Fq}$, and $0\leq
i\leq 2$, let
$$
T^i_{\sheaf{K}}(\sheaf{F})=R^ip_{1,!}(p_2^*\sheaf{F}\otimes\sheaf{K}).
$$
\par
Then $T^i_{\sheaf{K}}(\sheaf{F})$ is constructible and there exists an integer $A\geq 1$ such that for any middle-extension
sheaf $\sheaf{F}$ on $\Aa^1_{\Fq}$, and $0\leq i\leq 2$, we have
$$
\cond(T^i_{\sheaf{K}}(\sheaf{F}))\leq
(2\cond(\sheaf{K})\cond(\sheaf{F}))^A.
$$
\end{theorem}

In particular, if $f$ is obtained by reduction modulo $p$ of a fixed
non-constant rational function $f_1/f_2$, where $f_i\in \Zz[X,Y]$, and
if we have some integer $M\geq 1$ and, for each $p$, we consider a
sheaf $\sheaf{F}_p$ modulo $p$ with conductor $\leq M$, then we have
$$
\cond(T^1_{\sheaf{K}}\sheaf{F}_p)\ll 1
$$
for all primes. This allows us to apply all our estimates for trace
functions to the trace functions of these sheaves; we give some
examples in Section~\ref{sec-examples}.

\begin{remark}

\begin{enumerate}
\item The fact that the sheaf $T^i_{\sheaf{K}}(\sheaf{F})$ is
  constructible for any constructible sheaf $\sheaf{F}$ and all $i$
  follows from~\cite[Arcata, IV, Th. 6.2]{deligne} (see
  also~\cite[Th. 7.8.1]{leifu}.)
\item Note that we omitted the Tate twist in this statement, since it
  concerns purely geometric facts.
\end{enumerate}
\end{remark}

We need to consider all the transforms $T^i_{\sheaf{K}}$, and not only
$T^1_{\sheaf{K}}$ because this will turn out to be useful in the
proof, which is interleaved with the proof of the following other
useful fact:

\begin{theorem}[Bounds for Betti numbers]\label{th-betti}
  Let $\Fq$ be a finite field of order $q$ and characteristic $p$,
  $\ell$ a prime distinct from $p$. Let $\sheaf{K}$ be an $\ell$-adic
  sheaf on $\Aa^1\times \Aa^1$ over $\Fq$ of the form
$$
\sheaf{K}=\sheaf{L}_{\psi(f(X,Y))},
$$
where $\psi$ is a non-trivial additive $\ell$-adic character and
$f\in\Fq(X,Y)$ is a rational function with conductor $<p$. 
\par
There exists an integer $B\geq 1$ such that for any middle-extension
sheaf $\sheaf{F}$ on $\Aa^1_{\Fq}$ and for $0\leq i\leq 4$, we have
$$
\dim H^i_c(\Aa^2\times\bFq,p_2^*\sheaf{F}\otimes\sheaf{K})\leq
(2\cond(f)\cond(\sheaf{F}))^B.
$$
\end{theorem}

Roughly speaking, we will proceed as follows: 
\begin{enumerate}
\item we prove Theorem~\ref{th-conductor} for $\sheaf{F}$ the trivial
  sheaf, and observe that Theorem~\ref{th-betti} is a known fact in
  that case, from bounds on Betti numbers due to Bombieri,
  Adolphson-Sperber and Katz~\cite{katz-betti};
\item using Theorem~\ref{th-conductor} for the trivial sheaf, we first
prove Theorem~\ref{th-betti} for all input sheaves $\sheaf{F}$ and
$i=2$;
\item finally, we prove Theorem~\ref{th-conductor} in general and
deduce Theorem~\ref{th-betti} for all $i$.
\end{enumerate}
\section{Diophantine motivation of the proof}
\label{sec-motivation}

The arguments of the proof of Theorem~\ref{th-conductor} are purely
algebraic and geometric, and exercise much of the basic formalism of
étale cohomology, as well as a simple use of spectral
sequences. However, there is a concrete analytic motivation from
(expected) properties of sums of trace functions, and we will first
present it. This is based on the Riemann Hypothesis over finite
fields, and is similar in principle to the discussion~\cite[Lecture
IV, Interlude]{katz-four-lectures} by Katz that motivates the crucial
step in his paper.
\par
The first ingredient is a lemma that, essentially, allows one to
estimate, in terms of accessible global invariants, the conductor of a
middle-extension sheaf, satisfying some conditions, assuming one
already knows estimates for the rank and the number of
singularities. In other words, it provides a bound for the sum of Swan
conductors in global terms, assuming that the rank and number of
singularities are under control.
\par
To be slightly more precise, assume that $\sheaf{F}$ is a
middle-extension sheaf on $\Aa^1_{\Fq}$ which is pointwise pure of
weight $0$, and assume in addition the following conditions:
\par
(1) $\sheaf{F}$ has no geometrically trivial Jordan-H\"older factor;
\par
(2) the Frobenius action on $H^1_c(\bar{U},\sheaf{F})$ is \emph{pure}
of weight $1$, for the maximal dense open set $U$ on which $\sheaf{F}$
is lisse. 
\par
We then define the invariant
$$
\tilde{\sigma}(\sheaf{F})=\limsup_{\nu \ra +\infty}
\frac{|S_{\nu}(\sheaf{F})|}{q^{\nu/2}},
$$
where
$$
S_{\nu}(\sheaf{F})=\sum_{x\in
  U(\Fqn)}\frfn{\sheaf{F}}(x,q^{\nu}),
$$
for $\nu\geq 1$ (in other words, these are the sums of trace functions
over extension fields). Then we have
$$
\cond(\sheaf{F})\leq 3\rank(\sheaf{F})+n(\sheaf{F})
+\tilde{\sigma}(\sheaf{F}).
$$
\par
Indeed, using (1), the Lefschetz trace formula applied to $U$ over
$\Fqn$ gives
$$
S_{\nu}(\sheaf{F})=-\Tr(\frob^{\nu}\mid H^1_c(\bar{U},\sheaf{F})),
$$
so that the purity assumption implies
$$
\tilde{\sigma}(\sheaf{F}) =\dim H^1_c(\bar{U},\sheaf{F})
%%=\sigma(\sheaf{F}),
$$
and then the stated bound follows from Lemma~\ref{lm-cd-bound} below
(which is an elementary application of the Euler-Poincar\'e formula.)
\par
We now consider the situation of Theorem~\ref{th-conductor}. We will
assume (and this is where the argument is not easy to make rigorous in
a decent generality) that the sheaves
$\sheaf{G}=T_{\sheaf{K}}(\sheaf{F})$ whose conductor we wish to
control always satisfy the conditions above (i.e., that they are
middle-extensions, pointwise of weight $0$, and (1), (2) hold.)  We
first assume that we can find suitable estimates of the rank, of the
number of singularities, and of the punctual part of $\sheaf{G}$
(intuitively, this is possible because these amounts to fiber-by-fiber
considerations, which boil down to properties of one-variable sheaves,
which are relatively well-understood; the case of the trivial sheaf
$\sheaf{F}$ is quite elementary, but the details will turn out to be a
bit involved in the general case). We then need to estimate
$\tilde{\sigma}(\sheaf{G})$. For this purpose, we proceed in two
steps.
\par
In Step 1, we consider only the trivial input sheaf
$\sheaf{F}=\bQl$. We then have
\begin{equation}\label{eq-combine}
\frac{S_{\nu}(\sheaf{G})}{q^{\nu/2}}=-\frac{1}{q^{\nu}}
\sum_{x\in\Ff_{q^{\nu}}} \Bigl( \sum_{y\in\Ff_{q^{\nu}}}
\psi_{\nu}(f(x,y))\Bigr)= -\frac{1}{q^{\nu}}
\sum_{(x,y)\in\Ff_{q^{\nu}}\times\Ff_{q^{\nu}} }\psi_{\nu}(f(x,y))
\end{equation}
and the two-variable character sum (under Assumption (2) for
$\sheaf{G}$) has square-root cancellation, so that the bounds on Betti
numbers of~\cite{katz-betti} (or often their predecessors, due to
Bombieri and Adolphson-Sperber) give
$$
\limsup_{\nu\ra+\infty}
\frac{|S_{\nu}(\sheaf{G})|}{q^{\nu/2}}\leq C
$$
where $C\geq 1$ depends only on the conductor of $f$.
\par
In Step 2, we handle the case of a general sheaf $\sheaf{F}$. We then
have
\begin{equation}\label{eq-exchange}
\frac{S_{\nu}(\sheaf{G})}{q^{\nu/2}}=-\frac{1}{q^{\nu}} \sum_{x\in\Ff_{q^{\nu}}}
\sum_{y\in\Ff_{q^{\nu}}} \frfn{\sheaf{F}}(y,q^{\nu})\psi_{\nu}(f(x,y))=
-\frac{1}{q^{\nu}} \sum_{y\in\Ff_{q^{\nu}}}\frfn{\sheaf{F}}(y,q^{\nu}) 
\sum_{x\in\Ff_{q^{\nu}}}
\psi_{\nu}(f(x,y)).
\end{equation}
\par
The basic point is that this is the inner-product of the trace
functions of the dual sheaf of $\sheaf{F}$ and of the sheaf
$R^1p_{2,!} \sheaf{L}_{\psi(f(X,Y))}$.  This last sheaf, \emph{by the
  first step} (applied to $\sheaf{L}_{\psi(f(Y,X))}$), has conductor
bounded by a constant depending only on the conductor of $f$.  By
assumption again, we have square-root cancellation in this sum as
$\nu\ra +\infty$, and by the quasi-orthogonality formulation of
Deligne's proof of the Riemann Hypothesis over finite
fields~\cite{weilii}, we obtain
$$
\tilde{\sigma}(\sheaf{G})=\limsup_{\nu\ra +\infty}
\frac{|S_{\nu}(\sheaf{G})|}{q^{\nu/2}}\leq C',
$$
where $C'$ depends only on the conductors of $\sheaf{F}$ and of $f$.

\begin{remark}
  In terms of linear operators and of the standard (unnormalized)
  inner-product on functions on $\Fq$, we exploit the obvious identity
$$
\sum_{x\in\Fq}{(\kop_K\varphi)(x)}=\langle \kop_K\varphi,1\rangle=
\langle \varphi,\kop_K^*1\rangle,
$$
where the adjoint operator $\kop_K^*$ has kernel
$K^*(x,y)=\overline{K(y,x)}$; the first step in our sketch amounts to
bounding (the complexity of) $\kop^*_K1$, and the second applies
standard inequalities to deduce a bound for the sum over $x$.
\end{remark}

In contrast with this sketch, the proof of Theorem~\ref{th-conductor}
below is entirely algebraic and does not require the Riemann
Hypothesis over finite fields. It also applies in greater generality,
so that the assumptions (1) and (2) are not needed. Roughly speaking,
instead of sums of trace functions, we control directly the dimension
$\tilde{\sigma}(\sheaf{G})$ of $H^1_c(\bar{U},\sheaf{G})$ for the
transformed sheaf $\sheaf{G}$. The ``combination of sums''
in~(\ref{eq-combine}) and the ``exchange of order of summation''
in~(\ref{eq-exchange}) are replaced by arguments based on spectral
sequences (compare again with~\cite[Lecture IV,
Interlude]{katz-four-lectures}, and the dictionary~\cite[Sommes Trig.,
\S 2]{deligne}.) The proof is however complicated by the fact that we
must also control the possible punctual part of the transformed sheaf.
\par
Before giving the proof, we will present some algebraic preliminaries
and then discuss first the motivating applications in
Section~\ref{sec-examples} (Section~\ref{sec-prelims} may be skipped
in a first reading, since Section~\ref{sec-examples} will only refer
to it incidentally). We then set up the proof in
Section~\ref{sec-setup}, and follow by presenting an (almost)
self-contained account of the Fourier transform and of the special
case which is relevant to the \textsc{Polymath8} project as it appears in \cite{  polymath8ANT} (see
Section~\ref{sec-polymath}). Finally, we give the full proof of
Theorems~\ref{th-conductor} and~\ref{th-betti}.

\section{Preliminaries}\label{sec-prelims}

We begin by some preliminary inequalities between the dimensions of the  cohomology groups and the conductor of sheaves on the affine line.
\subsection{General results from \'Etale cohomology}

We first state formally some properties of \'etale cohomology that we
will often use.

\begin{proposition}\label{pr-etale}
  \emph{(1)} Let $f\,:\, Y_k\lra X_k$ be a morphism of algebraic
  varieties over an algebraically closed field $k$, with fibers of
  dimension $\leq n$.  Let $\sheaf{F}$ be a constructible $\ell$-adic
  sheaf on $Y$. We have $R^if_!\sheaf{F}=0$ for $i<0$ and for
  $i>2n$. In particular, if $\sheaf{F}$ is a sheaf on $X$ and $X$ has
  dimension $\leq n$, we have $H^i_c(X,\sheaf{F})=0$ for $i<0$ and for
  $i>2n$.
  \par
\emph{(2)} Let $X_k$ be an algebraic variety over an algebraically
closed field $k$, let $U\subset X$ be an open subset and $C=X-U$ its
complement. Let $\sheaf{F}$ be a constructible $\ell$-adic sheaf on
$X$. We have a long sequence  
\begin{equation}\label{eq-excision}
\cdots \lra H^i_c(U,\sheaf{F})\lra H^i_c(X,\sheaf{F})\lra
H^i_c(C,\sheaf{F})
\lra H^{i+1}_c(U,\sheaf{F})\lra \cdots,
\end{equation}
and in particular, for all $i\geq 0$, we have
\begin{gather}\label{eq-sigma}
  \dim H^i_c(X,\sheaf{F})\leq \dim H^i_c(U,\sheaf{F})+
  \dim H^i_c(C,\sheaf{F})\\
  \dim H^i_c(U,\sheaf{F})\leq \dim H^i_c(X,\sheaf{F})+
  \dim H^{i-1}_c(C,\sheaf{F}).
\label{eq-sigma1}
\end{gather}
\emph{(3)} Let $X_k$ be a smooth affine algebraic variety over an
algebraically closed field $k$, pure of dimension $n\geq 0$, and let
$\sheaf{F}$ be a lisse $\ell$-adic sheaf on $X$. We have
\begin{equation}\label{eq-affine-vanishing}
H^i_c(X,\sheaf{F})=0
\end{equation}
for $0\leq i<n$. 
\par
\emph{(4)} Let $f\,:\, X_k\lra Y_k$ be a morphism of algebraic
varieties over an algebraically closed field $k$, and let $\sheaf{F}$
be an $\ell$-adic constructible sheaf on $X$. Then, for $y\in Y$ and
$i\geq 0$, the stalk of $R^if_!\sheaf{F}$ at $y$ is naturally
isomorphic to $H^i_c(f^{-1}X,\sheaf{F})$.
\end{proposition}

\begin{proof}
  (1) is the cohomological dimension property; the vanishing of
  $R^if_!\sheaf{F}$ for $i<0$ is immediate by definition, while the
  vanishing for $i>2n$ can be found, e.g., in~\cite[Arcata, IV,
  Th. 6.1]{deligne} or~\cite[th. 7.4.5]{leifu}; the case of $H^i_c$
  follows by considering $f\,:\, X\lra \spec(k)$, the structure
  morphism.
\par
(2) is the so-called ``excision'' long-exact sequence, see for
instance~\cite[Sommes Trig., (2.5.1)*]{deligne}); the
inequality~(\ref{eq-sigma}) for a given $i\geq 0$ is an immediate
consequence of the fragment
$$
H^i_c(U,\sheaf{F})\lra H^i_c(X,\sheaf{F})\lra
H^i_c(C,\sheaf{F}),
$$
and~(\ref{eq-sigma1}) is a consequence of
$$
H^{i-1}_c(C,\sheaf{F})\lra H^i_c(U,\sheaf{F})\lra H^i_c(X,\sheaf{F}).
$$
\par
(3) is the property of affine cohomological dimension for lisse
sheaves; it follows for instance from the Poincar\'e duality
$$
H^i_c(X,\sheaf{F})\simeq H^{n-i}(X,\sheaf{F}^{*})
$$
where $\sheaf{F}^*$ is the dual of $\sheaf{F}$ (see for
instance~\cite[Sommes Trig., Remarque 1.18 (c)]{deligne}; note that
the right-hand side is a cohomology group with no restriction of
compact support) and the vanishing property
$$
H^i(X,\sheaf{F})=0
$$
for an affine scheme $X$ and $i<\dim(X)$ (see, e.g.,~\cite[Arcata, IV,
Th. 6.4]{deligne}).
\par
(4) is a special case of the proper base change theorem, (see,
e.g.,~\cite[Arcata, IV, Th. 5.4]{deligne} or \cite[Th. 7.4.4
(i)]{leifu}).
\end{proof}

The following lemma will also be used frequently:

\begin{lemma}\label{lm-tensor}
  Let $\sheaf{F}$ and $\sheaf{G}$ be middle-extension $\ell$-adic
  sheaves on $\Aa^1_{\Fq}$. Then we have
$$
H^0_c(\Aa^1\times\bFq,\sheaf{F}\otimes\sheaf{G})=0,
$$
i.e., the tensor product has no punctual part.
\end{lemma}

\begin{proof}
  In general, for a constructible sheaf $\sheaf{H}$ lisse on a dense
  open set $U\subset \Aa^1$, the condition
$$
H^0_c(\Aa^1\times\bFq,\sheaf{H})=0
$$
amounts to saying that, for all $x\in (\Aa^1-U))(\bFq)$, the
specialization map
$$
\sheaf{H}_{x}\lra \sheaf{H}_{\bar{\eta}}^{I_{x}}
$$
is injective (see~\cite[\S 4.4]{katz-sommes} for instance), where
$I_{\bar{x}}$ is the inertia group at $x$. We now have
$$
(\sheaf{F}\otimes\sheaf{G})_x=\sheaf{F}_x\otimes\sheaf{G}_x
\injecte \sheaf{F}_{\bar{\eta}}^{I_x}\otimes \sheaf{G}_{\bar{\eta}}^{I_x}
\subset (\sheaf{F}_{\bar{\eta}}\otimes\sheaf{G}_{\bar{\eta}})^{I_x}=
(\sheaf{F}\otimes\sheaf{G})_{\bar{\eta}}^{I_x}.
$$
\end{proof}
\subsection{Basic bounds on the dimension of cohomology groups}

Another frequently-used fact, which is implicit in our previous work
in the case of middle-extension sheaves, is the control of Betti
numbers of constructible sheaves on $\Aa^1$ in terms of the conductor:

\begin{lemma}\label{lm-cont-hic}
Let $\Fq$ be a finite field of characteristic $p$, let $\ell\not=p$
be a prime number and $\sheaf{F}$ an $\ell$-adic constructible sheaf
on $\Aa^1_{\Fq}$. For $i=0$, $2$, we have
$$
\dim H^i_c(\Aa^1\times\bFq,\sheaf{F})\leq \cond(\sheaf{F})
$$
and 
$$
\dim H^1_c(\Aa^1\times\bFq,\sheaf{F})\leq
2\cond(\sheaf{F})+\cond(\sheaf{F})^2. 
$$
\end{lemma}

\begin{proof}
  For $i=0$, this is obvious from the definition of
  $\pct(\sheaf{F})\leq\cond(\sheaf{F})$. For $i=2$, we use the fact
  that if $\sheaf{F}$ is lisse on a dense open subset $U\subset
  \Aa^1$,  we have
$$
H^2_c(\Aa^1\times\bar{\Fq},\sheaf{F})= H^2_c(\bar{U},\sheaf{F})\simeq
(\sheaf{F}_{\bar{\eta}})_{\pi_1(\bar{U},\bar{\eta})},
$$
the coinvariant space for the action of the geometric fundamental
group on the geome\- tric generic fiber (see, e.g.,~\cite[Sommes Trig.,
Rem. 1.18 (d)]{deligne}; the first equality is also a consequence of
excision) and hence
$$
\dim H^2_c(\Aa^1\times\bFq,\sheaf{F}) \leq
\rank(\sheaf{F})\leq \cond(\sheaf{F}).
$$
\par
For $i=1$, we use the Euler-Poincar\'e formula (see~\cite[8.5.2,
8.5.3]{katz-gkm}) to get
\begin{multline}\label{eq-euler-poincare-line}
  \dim H^1_c(\Aa^1\times\bFq,\sheaf{F}) =-\rank(\sheaf{F})+\dim
  H^0_c(\Aa^1\times\bFq,\sheaf{F}) +\dim
  H^2_c(\Aa^1\times\bFq,\sheaf{F})
  \\+\sum_{x}(\drop_x(\sheaf{F})+\swan_x(\sheaf{F}))+\swan_{\infty}(\sheaf{F})
\end{multline}
where the sum is over $x\in \Aa^1(\bFq)$, and all but finitely many
terms are zero, and the result follows from the definition of the
conductor since $\drop_x(\sheaf{F})=\rank(\sheaf{F})-\dim
\sheaf{F}_x\leq \rank(\sheaf{F})$.
\end{proof}

The following was also proved for middle-extensions in our previous
works.

\begin{lemma}\label{lm-tensor-cond}
  Let $\Fq$ be a finite field of characteristic $p$, let
  $\ell\not=p$ be a prime number and $\sheaf{F}_1$ and $\sheaf{F}_2$
  be $\ell$-adic constructible sheaves on $\Aa^1_{\Fq}$. We have
$$
\cond(\sheaf{F}_1\otimes\sheaf{F}_2)\leq
8\cond(\sheaf{F}_1)^2\cond(\sheaf{F}_2)^2
$$
\end{lemma}

\begin{proof}
  One checks easily (as in~\cite[Prop. 8.2 (2)]{FKM1}) that for the
  middle-extension part $(\sheaf{F}_1\otimes\sheaf{F}_2)_0$ we have
$$
\cond((\sheaf{F}_1\otimes\sheaf{F}_2)_0)\leq
6\cond(\sheaf{F}_1)^2\cond(\sheaf{F}_2)^2,
$$
and as for the punctual part, we have
$$
\pct(\sheaf{F}_1\otimes\sheaf{F}_2)\leq (n_1+n_2)m_1m_2
$$
where $n_i$ is the number of points where there are punctual sections
of $\sheaf{F}_i$, while $m_i$ is the maximal dimension of the space of
sections supported at a single point. Since
$$
(n_1+n_2)m_1m_2\leq (\cond(\sheaf{F}_1)+\cond(\sheaf{F}_2))m_1m_2
\leq 2\cond(\sheaf{F}_1)\cond(\sheaf{F}_2)m_1m_2
\leq 2\cond(\sheaf{F}_1)^2\cond(\sheaf{F}_2)^2, 
$$
we get the result.
\end{proof}

\subsection{Number of singularities}

We will also use a criterion to bound the number of singularities in
terms of estimates for the punctual part.

\begin{lemma}\label{lm-criterion-constant}
  Let $\Fq$ be a finite field of characteristic $p$, $\ell\not=p$ a
  prime number and $\sheaf{F}$ an $\ell$-adic constructible sheaf on
  $\Aa^1_{\Fq}$. Let $U\subset \Aa^1$ be a dense open set such that
  the dimension of the stalks $\sheaf{F}_x$ is constant, equal to some
  integer $d\geq 0$, for all $x\in U(\bFq)$. We then have
$$
n(\sheaf{F})\leq |(\Pp^1-U)(\bFq)|+\pct(\sheaf{F}).
$$
\end{lemma}

\begin{proof}
Since $U$ contains the generic point $\eta$ of $\Aa^1$, we have
$$
\rank(\sheaf{F})=\dim \sheaf{F}_{\bar{\eta}}=d.
$$
\par
Let $U_1\subset U$ be the open dense subset where $\sheaf{F}$ is
lisse, and let $x\in (U-U_1)(\bFq)$, i.e., a point of $U$ where
$\sheaf{F}$ is not lisse. Let $\varphi\,:\, \sheaf{F}_x\lra
\sheaf{F}_{\bar{\eta}}^{I_x}$ be the canonical map. The image has
dimension $<d$ (since otherwise, for dimension reasons, $I_x$ would
act trivially on the geometric generic fiber $\sheaf{F}_{\bar{\eta}}$,
and $\sheaf{F}$ would be lisse at $x$), and since $\dim\sheaf{F}_x=d$,
it follows that
$$
\dim \ker\varphi\geq 1,
$$
which means that $x$ is in the support of the punctual part of
$\sheaf{F}$. Thus the number of such $x$ is at most the size of this
support, which is bounded by $\pct(\sheaf{F})$. Adding the points of
$(\Pp^1-U)(\bFq)$ leads to the result.
\end{proof}

\subsection{Application to Artin-Schreier sheaves}

Given  $g\in\Fq(X)$ a non-constant rational function in one variable let $\mcL=\sheaf{L}_{\psi(g(X))}$
 its associated Artin-Schreier sheaf. The next lemma recall the crucial link between the Swan conductor of $\mcL$ at a  given point and its order as a pole of  $g$.

\begin{lemma}\label{lm-as1}
  Let $\Fq$ be a finite field of order $q$ and characteristic $p$,
  $\ell\not=p$ a prime number.  Let $g\in\Fq(X)$ be a non-constant rational function and $\sheaf{L}$ its associated Artin-Schreier sheaf. For $x\in
  \Pp^1(\bFq)$, the Swan conductor of $\sheaf{L}$ at $x$ is at most
  equal to the order of the pole of $g$ at $x$, and there is equality
  if the numerator and denominator of $g$ have degree $<p$.
\end{lemma}

\begin{proof}
This is a standard property (see, for instance,~\cite[Sommes
  Trig., (3.5.4)]{deligne}).
\end{proof}

We next discuss relations between two-variable Artin-Schreier sheaves
and specializations of one variable. We need first some notation.

\begin{definition}[Specializations]\label{def-spec}
  Let $\Fq$ be a finite field of order $q$ and characteristic $p$,
  $\ell$ a prime distinct from $p$. Let $f\in\Fq(X,Y)$ be a
  non-constant rational function and let
  $$\sheaf{K}:=\sheaf{L}_{\psi(f(X,Y))}$$ be the Artin-Schreier
sheaf on $\Aa^2_{\Fq}$ associated to $f$. We introduce the further notations:
\par
(1) If $x\in \bFq$ is such that $X-x$ does not divide the denominator
of $f$, we denote by $f_x\in \Fq(Y)$ the specialization $f(x,Y)$ of $f$.
\par
(2) For every finite extension
$k/\Fq$ and every $x\in k$, the \emph{specialization of $\sheaf{K}$ at
  $x$} is the $\ell$-adic constructible sheaf on $\Aa^1_{k}$ given by
\begin{equation}\label{eq-def-lx}
\sheaf{K}_x=j_x^*\sheaf{K},
\end{equation}
where $j_x\,:\, \{x\}\times \Aa^1\injecte \Aa^2$ is the closed
immersion.
\end{definition}

These two definitions are related as follows:

\begin{lemma}\label{lm-as2}
  Let $\Fq$ be a finite field of order $q$ and characteristic $p$,
  $\ell\not=p$ a prime number.  Let
  $\sheaf{K}=\sheaf{L}_{\psi(f(X,Y))}$ be an $\ell$-adic
  Artin-Schreier sheaf on $\Aa^2$ over $\Fq$, where $f\in\Fq(X,Y)$ is
  a non-constant rational function. 
\par
\emph{(1)} For any finite extension $k/\Fq$ and $x\in k$, we have
$$
\sheaf{K}_x=0
$$
if $X-x$ divides the denominator of $f$, and otherwise
$$
\sheaf{K}_x=j_!\sheaf{L}_{\psi(f_x)}
$$
where $j\,:\, U_x\lra \{x\}\times\Aa^1$ is the open immersion of the
open subset of $\{x\}\times\Aa^1$ which is the intersection of
$\{x\}\times\Aa^1$ and the open set of $\Aa^2$ where the denominator
of $f$ is invertible. 
\par
If $\{x\}\times\Aa^1$ does not intersect the zero set of the numerator
of $f$, then $\sheaf{K}_x$ is isomorphic to the Artin-Schreier sheaf
$\sheaf{L}_{\psi(f_x)}$ associated to $f_x$.
\par
\emph{(2)} For every finite extension $k/\Fq$ and all $x\in k$, we
have
$$
\cond(\sheaf{K}_x)\leq 2\cond(f).
$$
\end{lemma}

\begin{proof}
  (1) If $X-x$ divides the denominator of $f$, then by definition the
  sheaf $\sheaf{K}$ is zero on $\{x\}\times\Aa^1$, and hence
  $\sheaf{K}_x=0$.
\par
If $X-x$ does not divide the denominator of $f$, then there are only
finitely many points where $\{x\}\times\Aa^1$ intersects the open set
$U$ where the denominator is invertible. The sheaf $\sheaf{K}_x$ has
zero stalk at these points, and is isomorphic to the one-variable
sheaf $\sheaf{L}_{\psi(f(x,Y))}$ on the complementary open set, which
is the result we claim. 
\par
If $\{x\}\times\Aa^1$ does not intersect the zero set of the numerator
of $f$, then the points in $\{x\}\times\Aa^1$ where $\sheaf{K}_x$ has
zero stalk are precisely the poles of $f_x$, which means that
$j_!\sheaf{L}_{\psi(f_x)}=\sheaf{L}_{\psi(f_x)}$ as Artin-Schreier
sheaf on $\Aa^1_{\Fq}$.
\par
(2) If $\sheaf{K}_x=0$, then the conductor bound is trivial, and
otherwise we obtain from (1) the bound
$$
\cond(\sheaf{K}_x)\leq 1+\deg_Y f(x,Y)+\sum_{y\in\Pp^1(\bFq)}
\ord_y(f(x,Y)) \leq 1 +2\deg_Y f(x,Y)\leq 2\cond(f),
$$
as claimed.
\end{proof}

\begin{remark}
  Note that $\sheaf{K}_x$ is not always isomorphic to the
  Artin-Schreier sheaf $\sheaf{L}_{\psi(f_x)}$ on $\Aa^1$: for
  instance, if $f=X/Y$ and $x=0$, we have
  $\sheaf{L}_{\psi(f(x,Y))}=\bQl$, but $\sheaf{K}_0=j_!\bQl$, where
  $j\,:\, \Aa^1-\{0\}\injecte \Aa^1$ is the open immersion. Thus
  $\sheaf{K}_x$ has zero stalk at $0$. However, this subtlety will not
  be a problem for us, in particular because the set of $x$ for which
  this behavior happens (and the set of $y$ such that the stalk of
  $\sheaf{K}_x$ at $y$ is not the same as that of
  $\sheaf{L}_{\psi(f(x,Y))}$) is finite and -- since these points must
  be common zeros of the numerator $f_1$ and the denominator $f_2$ of
  $f$ -- has size bounded by $\deg(f_1)\deg(f_2)$, e.g. by Bezout's
  theorem.
\end{remark}

In particular, we get the following corollary from the previous four
lemmas. The statement uses the notation~(\ref{eq-def-lx}).

\begin{corollary}\label{cor-fiber-dim}
  Let $\Fq$ be a finite field of order $q$ and characteristic $p$,
  $\ell$ a prime distinct from $p$. Let
  $\sheaf{L}=\sheaf{L}_{\psi(f(X,Y))}$ be an $\ell$-adic
  Artin-Schreier sheaf on $\Aa^2$ over $\Fq$, where $f\in\Fq(X,Y)$ is
  a non-constant rational function. Let $\sheaf{F}$ be a
  middle-extension $\ell$-adic sheaf on $\Aa^1$ over $\Fq$.
\par
For every $x\in\Aa^1(\bFq)$, we have
$$
\dim H^1_c(\Aa^1\times\bFq,\sheaf{F}\otimes\sheaf{L}_x)\leq
3\cdot 2^{10}\cond(f)^4\cond(\sheaf{F})^4.
$$
\end{corollary}

\begin{proof}
  Combine Lemma~\ref{lm-cont-hic}, Lemma~\ref{lm-tensor-cond} and
  Lemma~\ref{lm-as2}.
\end{proof}

\subsection{A global bound for the conductor}

We now come to the lemma which contains the first idea in the proof of
Theorem~\ref{th-conductor}: it allows us to replace the sum of Swan
conductors, in the definition of the conductor of a sheaf, by a global
invariant (the Euler-Poincaré characteristic) that is more accessible
to algebraic manipulations.

\begin{lemma}[Global conductor bound]\label{lm-cd-bound}
  Let $\sheaf{F}$ be a middle-extension sheaf on $\Aa^1_{\Fq}$.
  % which is
  % pointwise pure of weight $0$ and
  % lisse on some dense open set $U$
  We have
$$
\cond(\sheaf{F})\leq 3\rank(\sheaf{F}) -\chi(\sheaf{F})\leq
3\rank(\sheaf{F}) +h^1(\sheaf{F}).
$$
% where 
% $$
% \chi(\sheaf{F}):=\chi_c(\Aa^1_{\bar{\Fq}}\sheaf{F})=\dim
% H^0_c(\Aa^1_{\bar{\Fq}},\sheaf{F})\  -\dim H^1_c(\Aa^1_{\bar{\Fq}},\sheaf{F})+ \dim
% H^2_c(\Aa^1_{\bar{\Fq}},\sheaf{F})
% $$ 
% is the compactly supported Euler-Poincar\'e characteristic of
% $\sheaf{F}$ over $\Aa^1_{\bar{\Fq}}$ and
% $$
% \sigma(\sheaf{F})=\dim H^1_c(\Aa^1_{\bar{\Fq}},\sheaf{F}).
% $$
\end{lemma}

\begin{proof}
  Let $U$ be the maximal open set on which $\mcF$ is lisse; it is
  dense. Since~$\mcF$ is a middle-extension sheaf, we have
  $$
  \cond(\mcF)=\rank(\mcF)+\sum_{x\in(\Pp^1-U)(\bar{\Fq})}(1+\swan_x(\sheaf{F})).
  $$
%   and let
%     $$\chi_c(\bar{U},\sheaf{F})=\dim H^0_c(\bar{U},\sheaf{F})\ -\dim
%     H^1_c(\bar{U},\sheaf{F})+ \dim H^2_c(\bar{U},\sheaf{F})
% $$ 
% be the compactly supported Euler-Poincar\'e characteristic of
% $\sheaf{F}$ over $\bar{U}$. 

  By the Euler-Poincaré formula (see, e.g.,~\cite[2.3.1]{katz-gkm}),
  we have
$$
-\chi_c(U,\sheaf{F})= -\rank(\sheaf{F})\chi_c(\bar{U},\bQl)+
\sum_{x\in(\Pp^1-U)(\bar{\Fq})}\swan_x(\sheaf{F}).
$$
We have $\chi_c(U,\bQl)=2-|(\Pp^1-U)(\bFq)|=2-n(\mcF)$, and therefore
$$
\sum_x\swan_x(\sheaf{F})= -\chi_c(U,\sheaf{F})
+(2-n(\mcF))\rank(\sheaf{F}).
$$
By excision we have
$$
-\chi_c(U,\sheaf{F})=
-\chi(\sheaf{F})+\sum_{x\in(\Aa^1-U)(\bar{\Fq})}\dim\mcF_x\leq
-\chi(\sheaf{F})+\sum_{x\in(\Pp^1-U)(\bar{\Fq})}\dim\mcF_x
$$
and we
obtain the upper bound
$$
\cond(\mcF)\leq 3\rank(\mcF)
-\chi(\mcF)+\sum_{x\in(\Pp^1-U)(\bar{\Fq})}(1+\dim\mcF_x-\rank(\mcF))
$$
\par
Finally, since $\mcF$ is a middle extension sheaf, we have
$$
\dim\mcF_x\leq \rank(\sheaf{F})-1,
$$
for any $x\in (\Pp^1-U)(\bar{\Fq})$, because $\mcF$ is not lisse at
$x$. The first inequality follows, and also the second since
$$
-\chi(\sheaf{F})\leq \dim H^1_c(\Aa^1_{\bar{\Fq}},\sheaf{F}).
$$
\end{proof}

\section{Examples and applications}
\label{sec-examples}

\subsection{Preliminaries on trace functions}

The simplest applications of our results consist in plugging the trace functions of transform sheaves $T_{\sheaf{K}}(\sheaf{F})$ in any general result involving trace functions. 

One must be slightly careful since many results are stated for irreducible middle-extension sheaves which are
pointwise pure of some weight and the sheaf $T_{\sheaf{K}}(\sheaf{F})$ may not have these properties (in particular it may not be irreducible even if $\sheaf{F}$ is.) 
\par
There is a potential notational subtlety (which did not arise in our
previous works) involving the definition of weights. For an integer
$n\in\Zz$, recall (see~\cite[Def. 1.2.2]{weilii}) that an $\ell$-adic
sheaf $\sheaf{F}$ on $X_{\Fq}$ is \emph{pointwise pure of weight $n$}
if, for all finite extensions $k/\Fq$ and for all $x\in X(k)$, the
eigenvalues of Frobenius acting on $\sheaf{F}_{\bar{x}}$ are
$|k|$-Weil numbers of some weight $w=n$. A sheaf $\sheaf{F}$ on $X$ is
\emph{mixed of weights $\leq n$} if it has a finite filtration
$$
0=\sheaf{F}_0\subset \sheaf{F}_1\subset \cdots\subset
\sheaf{F}_m=\sheaf{F}
$$
where the successive quotients $\sheaf{F}_i/\sheaf{F}_{i-1}$ are
pointwise pure with weight $n_i\leq n$.
\par 
On the other hand (see~\cite[(7.3.7)]{katz-esde}), a middle-extension
sheaf $\sheaf{F}$ on a curve $X_{\Fq}$ is \emph{pure of weight $n$}
if, for some (equivalently any) dense open set $U\subset X$ where $\sheaf{F}$ is lisse,
for all $k/\Fq$ and all $x\in U(k)$, the eigenvalues of Frobenius on
$\sheaf{F}_{\bar{x}}$ are $|k|$-Weil numbers of weight $n$. It follows
from results of Deligne (in particular~\cite[Lemme 1.8.1]{weilii}, and
the Riemann Hypothesis) that such a sheaf is also mixed of weights
$\leq n$, i.e., the eigenvalues of Frobenius at the ``missing points''
$X-U$ are also Weil numbers with weight $\leq n$. However, these
weights may be $<n$. In other words, a middle-extension sheaf may be
pure of weight $n$ without being pointwise pure of weight $n$.
\par
The following lemma encapsulates a reduction of trace functions of
constructible sheaves to middle-extension sheaves:

\begin{lemma}[Trace function of constructible sheaf]
\label{lm-devisse}
Let $\Fq$ be a finite field of characteristic $p$, let $\ell\not=p$
be a prime and let $\sheaf{F}$ be an $\ell$-adic constructible sheaf
on $\Aa^1_{\Fq}$ which is mixed of weights $\leq 0$.
\par
There exists a decomposition of the trace function $\frfn{\sheaf{F}}$
of $\sheaf{F}$ of the form
$$
\frfn{\sheaf{F}}=\frfn{\sheaf{F}^{mid}}+t_1+t_2,
$$
where $\sheaf{F}^{mid}$ is a middle-extension sheaf on $\Aa^1_{\Fq}$
which is pure of weight $0$, and where:
\par
\emph{(1)} The function $t_1$ is zero except for a set of values of
$x\in\Fq$ of size at most $2\cond(\sheaf{F})$, and it satisfies
$$
|t_1(x)|\leq 2\cond(\sheaf{F})
$$
for all $x\in\Fq$.
\par
\emph{(2)} The function $t_2$ satisfies
$$
|t_2(x)|\leq \cond(\sheaf{F})q^{-1/2}
$$
for all $x\in\Fq$.
\end{lemma}

\begin{proof}
  This is a classical ``dévissage''. We begin by observing that
$$
|\frfn{\sheaf{F}}(x)|\leq \cond(\sheaf{F})
$$
for all $x\in\Fq$: indeed, by assumption, all eigenvalues of Frobenius
on the stalk $\sheaf{F}_{\bar{x}}$ are of modulus at most $1$, and the
maximal dimension of a stalk is bounded by the conductor (including
where there is a punctual part of the sheaf.)
\par
Let $\sheaf{F}^0$ be the direct sum of quotients which are pointwise
pure of weight $0$ in a filtration of $\sheaf{F}$ with successive
quotients which are pointwise pure of some weight $\leq 0$, and let
$\sheaf{F}^1$ be the direct sum of the remaining quotients. We have
$$
\frfn{\sheaf{F}}(x)=\frfn{\sheaf{F}^0}(x)+\frfn{\sheaf{F}^1}(x),
$$
and trivially
$$
|\frfn{\sheaf{F}^1}(x)|\leq p^{-1/2}\cond(\sheaf{F})
$$
for all $x\in\Fq$. We put $t_2=\frfn{\sheaf{F}^1}$. 
\par
Next, let
$$
0\lra \sheaf{F}^{pct}\ra \sheaf{F}^0\ra \sheaf{F}^{npct}\ra 0
$$
be the short exact sequence associated to the inclusion of the
punctual part $\sheaf{F}^{pct}$ of $\sheaf{F}^0$. We have
$$
\frfn{\sheaf{F}^0}(x)=\frfn{\sheaf{F}^{pct}}(x)+\frfn{\sheaf{F}^{npct}}(x),
$$
and $\frfn{\sheaf{F}^{pct}}$ is zero except for $\leq
\cond(\sheaf{F})$ values of $x$ for which we have
$$
|\frfn{\sheaf{F}^{pct}}(x)|\leq \dim
H^0_c(\Aa^1\times\bar{\Ff}_p,\sheaf{F}^0)\leq \cond(\sheaf{F}).
$$
\par
Finally, let $j\,:\, U\injecte \Aa^1$ be the open immersion of the
maximal dense open subset where $\sheaf{F}^{npct}$ is lisse, and let
$$
\sheaf{F}^{mid}=j_*j^*\sheaf{F}^{npct}.
$$
\par
This is a middle-extension sheaf, pointwise pure of weight $0$, with
trace function equal to that of $\sheaf{F}$ for $x\in U(\Fq)$. Thus
the difference
$$
\frfn{\sheaf{F}}-\frfn{\sheaf{F}^{mid}}
$$
is zero except for at most $\cond(\sheaf{F})$ values of $x\in \Fq$,
and has modulus $\leq 2\cond(\sheaf{F})$ for all $x$. We obtain the
desired decomposition by taking
$$
t_1=\frfn{\sheaf{F}^{pct}}+\frfn{\sheaf{F}}-\frfn{\sheaf{F}^{mid}}.
$$
\end{proof}

We will apply the previous lemma to the trace functions of the transform sheaves
$T^1_{\sheaf{K}}(\sheaf{F})$ considered in this paper. We introduce
a definition for convenience.

\begin{definition}[$f$-disjoint sheaf]
  Let $\Fq$ be a finite field of characteristic $p$, let $\ell\not=p$
  be a prime. Let $f\in\Fq(X,Y)$ be a rational function and let
  $\sheaf{K}=\sheaf{L}_{\psi(f)}$ be the Artin-Schreier sheaf on
  $\Aa^2_{\Fq}$ associated to $f$.
\par
A middle-extension sheaf $\sheaf{F}$ on $\Aa^1_{\Fq}$ is called
\emph{$f$-disjoint} or \emph{$\sheaf{K}$-disjoint} if for all $x\in\Aa^1(\bFq)$, one has
$$
H^2_c(\Aa^1\times\bFq,\sheaf{F}\otimes \sheaf{K}_x)=0.
$$
 
\end{definition}

\begin{corollary}[Artin-Schreier transforms as trace functions]
\label{cor-transform}
Let $\Fq$ be a finite field of characteristic $p$, let $\ell\not=p$ be
a prime. Let $f\in\Fq(X,Y)$ be a rational function given by
$f=f_1/f_2$ with $f_i\in \Fq[X,Y]$ coprime polynomials, and assume
that $\cond(f)<p$.
\par
Let $\sheaf{F}$ be a middle-extension sheaf on $\Aa^1_{\Fq}$ which is
pointwise pure of weight $0$ and $f$-disjoint.
\par
There exists an absolute constant $A\geq 1$, independent of $f$ and
$\sheaf{F}$, such that for all $x\in \Aa^1(\Fq)$, we have
$$
\frac{1}{\sqrt{q}}\sum_{\stacksum{y\in \Fq}{f_2(x,y)\not=0}}
\frfn{\sheaf{F}}(y)\psi(f(x,y))
=-t_0(x)+t_1(x)+t_2(x),
$$
where \begin{itemize}
 \item[-] 	 $t_0$ is the trace function of a middle-extension sheaf
$\sheaf{G}^{mid}$ of weight $0$ on $\Aa^1_{\Fq}$ with
$$
\cond(\sheaf{G}^{mid})\leq (2\cond(f)\cond(\sheaf{F}))^A,
$$
\item[-] the function $t_1$ is zero for a set of
values of $x\in\Fq$ of size at most $(2\cond(f)\cond(\sheaf{F}))^A$,
and it satisfies
$$
|t_1(x)|\leq (2\cond(f)\cond(\sheaf{F}))^A, 
$$
for all $x\in\Fq$,
\item[-] The function $t_2$ satisfies
$$
|t_2(x)|\leq (2\cond(f)\cond(\sheaf{F}))^Aq^{-1/2}
$$
for all $x\in\Fq$.
 \end{itemize}

\end{corollary}

\begin{proof}
Let
$$\sheaf{G}_i=T^i_{\sheaf{K}}(\sheaf{F})(1/2)\hbox{ for $0\leq i\leq 2$ and }
\sheaf{G}=\sheaf{G}_1=T^1_{\sheaf{K}}(\sheaf{F})(1/2).
$$
\par
By the Riemann Hypothesis~\cite{weilii} (taking into account the Tate
twist) the sheaves $\sheaf{G}_i$ are mixed of weight $\leq i-1$ and in particular $\sheaf{G}$ is mixed of weight $\leq 0$. By the proper base change theorem (see
Proposition~\ref{pr-etale}, (4)) and the Grothendieck-Lefschetz trace
formula, the trace function of
$\sheaf{G}$ is
$$
\frfn{\sheaf{G}}(x)=-\frac{1}{\sqrt{q}}\sum_{y\in\Fq}\frfn{\sheaf{F}}(y)
\frfn{\sheaf{K}}(x,y)+ \frfn{\sheaf{G}_0}(x)+\frfn{\sheaf{G}_2}(x)
$$
for $x\in \Fq$. 
\par
The stalk of $\sheaf{G}_0$ over $x$ is (by Lemma~\ref{lm-tensor}) 
$$
H^0_c(\Aa^1\times\bFq, \sheaf{F}\otimes \sheaf{L}_{x})=0
$$
and  that of $\sheaf{G}_2$ is (since $\sheaf{F}$ is $f$-disjoint )
$$
H^2_c(\Aa^1\times\bFq,\sheaf{F}\otimes \sheaf{L}_{x})=0.
$$
Hence we obtain, for all $x\in\Fq$,
$$
\frfn{\sheaf{G}}(x)=
-\frac{1}{\sqrt{q}}\sum_{y\in\Fq}\frfn{\sheaf{F}}(y)\frfn{\sheaf{K}}(x,y)
$$
By Definition~\ref{def-as}, we have
$$
\frfn{\sheaf{K}}(x,y)=
\begin{cases}
\psi(f(x,y))&\text{ if } f_2(x,y)\not=0\\
0&\text{ otherwise,}
\end{cases}
$$
and by Theorem~\ref{th-conductor}, there exists $A\geq 1$ such that
the constructible sheaf $\sheaf{G}$ satisfies
$$
\cond(\sheaf{G})\leq (2\cond(f)\cond(\sheaf{F}))^A.
$$
\par
Thus the result follows by applying Lemma~\ref{lm-devisse} to
$\sheaf{G}$. 
\end{proof}
\begin{definition}
	The middle extension sheaf
\begin{equation}\label{defwt0part}
\mcG^{mid}=T^1_{\sheaf{K}}(\sheaf{F})^{mid}(1/2)	
\end{equation}
obtained by applying Lemma \ref{lm-devisse} to $\mcG$ will be called the weight $0$ part of $\mcG$.
\end{definition}

\begin{remark}%\label{rm-good-inputs} 
  For $\sheaf{K}$ as in this corollary, the condition that $\sheaf{F}$
  is $f$-disjoint is valid in many cases. We list some of them for
  convenience. 
  
  The assumption of Corollary~\ref{cor-transform} holds:
\begin{enumerate}
	\item If $\sheaf{F}$ is irreducible of rank at least $2$ (e.g.,
Kloosterman sheaves in one or more variables), or more generally if
$\sheaf{F}$ is irreducible and not isomorphic to an Artin-Schreier
sheaf.
\par
\item If $\sheaf{F}$ is tamely ramified and there is no specialization
$f_x$ of $f$ which is constant as an element in $\Fq(Y)$ (e.g., any
Kummer sheaf with trace function $\chi(g(x))$ for a multiplicative
character $\chi$, provided no $f_x$ is constant); in particular, if
$\sheaf{F}$ is the trivial sheaf with constant trace function $1$, it
is enough that no specialization $f_x$ be constant.
\item If $\sheaf{F}$ is an Artin-Schreier sheaf $\sheaf{L}_{\psi(g)}$
with trace function $\psi(g(x))$ and there is no $x\in\bFq$ such that
$g+f_x$ is constant.

\end{enumerate}

\end{remark}

\subsection{Application to automorphic twists}

We begin by explaining one setting where the application of our result
is very easy:

\begin{proposition}
  Let $f$ be a Hecke cusp form of level $N\geq 1$ with Fourier
  coefficients $\rho_f(n)$ at $\infty$. Let $g_1$, $g_2\in\Zz[X,Y]$ be
  two non-constant coprime polynomials, and let
  $g=g_1/g_2\in\Qq(X,Y)$.
\par
Let $V$ be a smooth function on $]0,+\infty[$ with compact
support. Let $p$ be a prime number, let $K$ be an irreducible trace
function modulo $p$ associated to a middle-extension sheaf $\sheaf{F}$
which is $(g\mods{p})$-disjoint. For $\eps>0$, we have
$$
\sum_{n\geq 1} \rho_f(n)
\frac{1}{\sqrt{p}}\Bigl(\sum_{\stacksum{x\in\Fp}{g_2(n,x)\not=0\mods{p}}}K(x)
{e\Bigl(\frac{g_1(n,x)\overline{g_2(n,x)}}{p}\Bigr)}\Bigr)V(n/p) \ll
p^{1-1/8+\eps}
$$
where the implied constant depends on $(f,V,\eps,\cond(\mcF),\cond(g))$.
\end{proposition}

\begin{proof}
The main result of~\cite{FKM1} shows that
$$
\sum_{n\geq 1} \rho_f(n) K(n)V(n/p) \ll p^{1-1/8+\eps}
$$
if $K$ is the trace function of a geometrically isotypic
middle-extension sheaf which is pointwise pure of weight $0$. We
will show how to deduce the result from this.
\par
By Corollary~\ref{cor-transform} (applied with $\psi$ chosen so that
$\frfn{\sheaf{L}_{\psi}}(x)=e(x/p)$ for $x\in\Fp$), we have a
decomposition
$$
\frac{1}{\sqrt{p}}\Bigl(\sum_{\stacksum{x\in\Fp}{g_2(n,x)\not=0}}K(x)
{e\Bigl(\frac{g_1(n,x)\overline{g_2(n,x)}}{p}\Bigr)}\Bigr)= -t_0(n)+t_1(n)+t_2(n)
$$
where $-t_0$ is the trace function of a middle-extension sheaf which
is pure of weight $0$ and has conductor $\leq
C=(2\cond(f)\cond(\mcF))^A$, while $t_1$ is zero except for $\leq C$
values of $x\in\Fp$, where it has modulus at most $C$, while
$|t_2|\leq Cp^{-1/2}$. We have then
$$
\sum_{n\geq 1} \rho_f(n) t_i(n)V(n/p) \ll p^{1-1/8+\eps}
$$
for $i=1$, $2$, and we are reduced to the case of $t_0$. Decomposing
$t_0$ in trace functions of its geometrically isotypic components, we
conclude by applying~\cite{FKM1}.
\end{proof}

\subsection{Two-variable sums and the example of Conrey-Iwaniec}

A basic application of bounds on conductors like those of
Theorem~\ref{th-conductor} concerns two-variable exponential sums of
quite general type. We present the very general principle before
giving a concrete example. 
\par
Given a trace function $K(x,y)$ in two variables,
e.g. $$K(x,y)=\chi(f_1(x,y))e(f_2(x,y)/p)$$ for rational functions $f_1$
and $f_2\in \Ff_p(X,Y)$ and for a multiplicative character $\chi$
modulo $p$, one wishes to obtain square-root cancellation (when
possible) for
$$
\sum_{x,y} K(x,y).
$$
\par
This may be written as
$$
\sum_x \sum_{y}K(x,y),
$$
i.e., as the inner product of the constant function $1$ (i.e., the
trace function of the trivial sheaf) and (essentially) the trace
function of
$$
T^1_{\sheaf{K}}(\bQl)
$$
where $\sheaf{K}$ is the sheaf with trace function $K$. It may happen
that $K$ is given naturally as a product
$$
K(x,y)=K_1(x)K_2(y)K_3(x,y)
$$ 
for trace functions $K_1$ and $K_2$ modulo $p$ and another trace
function $K_3$ in two variables; in such a case, it may be better to
write the sum as
$$
\sum_x K_1(x)  \sum_{y}K_2(y)K_3(x,y),
$$
which is the inner-product of $K_1$ with the trace function of
$T^1_{\sheaf{K}_3}(\sheaf{K}_2)$, with obvious notation.
\par
From a direct application of the Riemann Hypothesis, we obtain the
following qualitative information concerning these types of sums:

\begin{proposition}[Small diagonal principle]\label{pr-principle}
  Let $\Fq$ be a finite field of characteristic $p$, let $\ell\not=p$
  be a prime number. Let $\sheaf{K}$ be a constructible $\ell$-adic
  sheaf mixed of weight $\leq 0$ on $\Aa^2_{\Fq}$.
\par
Let $\sheaf{F}_2$ be a middle-extension sheaf on $\Aa^1_{\Fq}$,
pointwise pure of weight $0$ such that $T^2_{\sheaf{K}}(\sheaf{F}_2)$
is generically $0$.
\par
There exists a finite set $X(\sheaf{K},\sheaf{F}_2)$ of geometrically
irreducible middle-extension sheaves which are pointwise pure of
weight $0$, of cardinality bounded in terms of the conductor of
$T^1_{\sheaf{K}}(\sheaf{F}_2)$, such that if $\sheaf{F}_1$ is a
middle-extension sheaf of weight $0$, geometrically irreducible, and
not geometrically isomorphic to any of the sheaves in
$X(\sheaf{K},\sheaf{F}_2)$, then
$$
\sum_{x,y\in\Fq}\frfn{\sheaf{F}_1}(x) \frfn{\sheaf{F}_2}(y)
\frfn{\sheaf{K}}(x,y)
\ll q,
$$
where the implied constant depends only on the conductor of
$\sheaf{F}_1$ and of $T^1_{\sheaf{K}}(\sheaf{F}_2)$.
\end{proposition}

\begin{proof}
  Let $X(\sheaf{K},\sheaf{F}_2)$ be the set of geometric isomorphism
  classes of geometrically irreducible components of the weight $0$
  part of $\mcG=T^1_{\sheaf{K}}(\sheaf{F}_2)(1/2)$. This is a finite set of
  cardinality bounded by the rank of $T^1_{\sheaf{K}}(\sheaf{F}_2)$,
  hence bounded in terms of the conductor of $\sheaf{F}_2$.
\par
Under the assumptions of the proposition, for $\sheaf{F}_1$
geometrically irreducible and not in $X(\sheaf{K},\sheaf{F}_2)$, we
have
\begin{align*}
  \frac{1}{q^{1/2}}\sum_{x,y\in\Fq}\frfn{\sheaf{F}_1}(x)
  \frfn{\sheaf{F}_2}(y) \frfn{\sheaf{K}}(x,y)&=
  \sum_{x\in\Fq}\frfn{\sheaf{F}_1}(x)
  \frac{1}{q^{1/2}}\sum_{y\in\Fq}\frfn{\sheaf{F}_2}(y)
  \frfn{\sheaf{K}}(x,y)\\&=-\Tr(\frob\mid
  H^1_c(\Aa^1\times\bar{\Ff}_q,\sheaf{F}_1\otimes \mcG))
\end{align*}
(since the cohomology spaces $H^0_c$ and $H^2_c$ vanishing here).  The
first cohomology space is mixed of weights~$\leq 1$ by Deligne's
Riemann Hypothesis, and hence we obtain the result using the conductor
bounds (Lemmas~\ref{lm-cont-hic} and~\ref{lm-tensor-cond}).
\end{proof}

Although this proposition does not, by itself, give square-root
cancellation in any individual case, it implies for instance that
$$
\sum_x e\Bigl(\frac{ax^2}{p}\Bigr) \sum_{y}\frfn{\sheaf{F}_2}(y)
\frfn{\sheaf{K}}(x,y)\ll p
$$
(working over $\Fp$) for all $a\in\Fp$ except for a number of
exceptions bounded in terms of the conductors of $\sheaf{F}_2$ and
$\sheaf{K}$ only. In quite a few applications, this type of
qualitative ``control of the diagonal'' is sufficient (for instance,
similar ideas are crucial in~\cite{FKM1}.) However, this is not always
the case, and one needs to attempt some further analysis if a more
precise result is needed.
\par
We now present a concrete example, taken from the important work of
Conrey and Iwaniec on the third moment of special values of
automorphic $L$-functions~\cite{CI}. Given a prime $p$ and two
multiplicative characters $\chi_1$ and $\chi_2$ modulo $p$, Conrey and
Iwaniec consider the sum
\begin{align*}
  S(\chi_1,\chi_2)&=\sum_{x,y\in\Fp} \chi_1(xy(x+1)(y+1))\chi_2(xy-1)\\
  &=\sum_{x\in\Fp}
  \chi_1(x(x+1))\sum_{x,y\in\Fp}\chi_1(y(y+1))\chi_2(xy-1).
\end{align*}
They prove:

\begin{theorem}[Conrey-Iwaniec]\label{th-ci}
Let $\chi_1$ be a non-trivial multiplicative character modulo $p$, and
let $\chi_2$ be any multiplicative character modulo $p$. Then
$$
S(\chi_1,\chi_2)\ll p
$$
where the implied constant is absolute.
\end{theorem}

This is~\cite[Lemma 13.1]{CI}, slightly generalized, since we do not
assume that $\chi_1$ is a real character. Conrey and Iwaniec
remark~\cite[Remarks, p. 1208]{CI} that their main result concerning
$L$-functions would be considerably weakened if (for $\chi_1$ a real
character modulo $p$, for many primes $p$) there existed a single
character $\chi_2$ for which the size of the sum would be $p^{3/2}$.

\begin{remark}
  Some natural generalizations of these sums have appeared recently in
  the work of Petrow and Young~\cite{PY} on the Weyl bound for
  Dirichlet $L$-functions. They prove the analogue of the theorem of
  Conrey and Iwaniec using some of the ideas in this paper.
\end{remark}

We will explain how to prove Theorem~\ref{th-ci} using the ideas of
cohomological transforms.  The sums $p^{-1/2}S(\chi_1,\chi_2)$ are
naturally presented in the form discussed above, namely as the inner
product of the trace function of the dual of the Kummer sheaf
\begin{equation}\label{defF1}
\sheaf{F}_1=\sheaf{L}_{\chi_1(X(X+1))}
\end{equation}
with that of the transform sheaf
\begin{equation}\label{defKummersheaf}
  \sheaf{G}=T^1_{\sheaf{K}}(\mcF_1)(1/2)=
  T^1_{\sheaf{K}}(\sheaf{L}_{\chi_1(Y(Y+1))})(1/2)\text{ where }
  \sheaf{K}=\sheaf{L}_{\chi_2(XY-1)}
\end{equation}
(the latter is defined as the extension by $0$ of the Kummer sheaf
$\sheaf{L}_{\chi_2(XY-1)}$ on the open set complement of the curve
$XY-1$, see below for the general definition.)
\par
More precisely, the trace function of $\sheaf{G}$ is 
$$
\frfn{\sheaf{G}}(x)=-\frac1{p^{1/2}}\sum_{y\in\Fp}\chi_1(y(y+1))\chi_2(xy-1)
$$
for all $x\in \Fp$, provided $\chi_1\not=1$: indeed, by the trace
formula and the proper base change theorem, it is enough to show that
$T^0_{\sheaf{K}}(\sheaf{L}_{\chi_1(Y(Y+1))})=
T^2_{\sheaf{K}}(\sheaf{L}_{\chi_1(Y(Y+1))})=0$ in that case. The
former is true by Lemma~\ref{lm-tensor}, and the latter because the
fiber above $x\in \bar{\Ff}_p$ is
$$
H^2_c(\Aa^1\times\bar{\Ff}_p,
\sheaf{L}_{\chi_1(Y(Y+1))}\otimes\sheaf{L}_{\chi_2(xY-1)})=0
$$
(since $\chi_1\not=1$, this can only be non-zero if the second tensor
factor is ramified at $0$ and $-1$, but it is in fact always
unramified at $0$.)
\par
The kernel $\sheaf{K}$ is not of the type considered in
Theorem~\ref{th-conductor}. However, it is easy to adapt the proof of
this result to derive an analogue for multiplicative characters. These
we define in general in analogy with Definition~\ref{def-as}:

\begin{definition}
  Given a non-trivial multiplicative $\ell$-adic character of
  $\Fq^{\times}$, we denote by $\mcL_\chi$ the associated Kummer sheaf
  on $\mathbf{G}_{m,\Fq}$ (see~\cite[Sommes Trig.]{deligne}).  Let
  $f\in\Fq(X,Y)$ be a rational function, $U$ the open set where the
  numerator and denominator are both non-zero, with
  $j\,:\, U\injecte \Aa^2$ the open immersion; let
  $f_U\,:\, U\lra \Gg_m$ be the associated morphism, then the {\em
    Kummer sheaf} associated to $f$ is the constructible $\ell$-adic
  sheaf on $\Aa^2_{\Fq}$ defined as
$$
\sheaf{L}_{\chi(f)}=j_!f_U^*\sheaf{L}_{\chi}.
$$
\end{definition}

\begin{theorem}[Conductor of Kummer transforms]%\label{th-kummer}
  Let $\Fq$ be a finite field of order $q$ and characteristic $p$,
  $\ell$ a prime distinct from $p$. Let $\sheaf{K}$ be an $\ell$-adic
  sheaf on $\Aa^1\times \Aa^1$ over $\Fq$ of the form
  $\sheaf{K}=\sheaf{L}_{\chi(f)}$.
\par
For constructible sheaves $\sheaf{F}$ on $\Aa^1_{\Fq}$, and $0\leq
i\leq 2$, let
$$
T^i_{\sheaf{K}}(\sheaf{F})=R^ip_{1,!}(p_2^*\sheaf{F}\otimes\sheaf{K}).
$$
\par
The sheaves $T^i_{\sheaf{K}}(\sheaf{F})$ are constructible and there
exists an absolute constant $A\geq 1$ such that
$$
\cond(T^i_{\sheaf{K}}(\sheaf{F}))\leq
(2\cond(\sheaf{K})\cond(\sheaf{F}))^A
$$
and moreover
$$
\dim H^i_c(\Aa^2\times\bFq,p_2^*\sheaf{F}\otimes\sheaf{K}) \leq
(2\cond(f)\cond(\sheaf{F}))^A.
$$
\end{theorem}

\begin{proof}[Sketch of proof]
  One can follow line by line the proof of Theorems~\ref{th-conductor}
  and~\ref{th-betti}. The only differences are:
\par
\begin{enumerate}
\item we deal separately with the case 
$$f(X,Y)=f_1(X)f_2(Y)$$
(instead of $f(X,Y)=f_1(X)+f_2(Y)$ in the case of an additive character);
\item 	to bound the Betti numbers
$$
\dim H^i_c(\Aa^2\times\bFq,\sheaf{K})
$$
(i.e., when the input sheaf is trivial), one uses the results of
Adolphson-Sperber or Katz~\cite[Th. 12]{katz-betti} instead of those
of Bombieri (which are only proved for additive characters); an alternative is to lift the tame sheaves to characteristic $0$.

\end{enumerate}

\end{proof}

In particular, in our case,
$\sheaf{G}=T^1_{\sheaf{K}}(\sheaf{L}_{\chi_1(Y(Y+1))})(1/2)$ has conductor
absolutely bounded as $\chi_1$, $\chi_2$ and $p$ vary.  By the Riemann
Hypothesis, the sheaf $\sheaf{G}$ is also mixed of weights $\leq 0$,
and therefore the principle above shows that, for all primes $p$, and
for all characters $\chi_2$, we have
$$
S(\chi_1,\chi_2)\ll p
$$
with an absolute implied constant, for all but a bounded number of
multiplicative characters $\chi_1$ modulo $p$ (since
$\sheaf{L}_{\chi_1(X(X+1))}\simeq \sheaf{L}_{\chi'_1(X(X+1))}$ if and
only if $\chi_1=\chi'_1$.)
\par
In order to go deeper and show that, in fact, these exceptions do not
exist, we must look more carefully at $\sheaf{G}$.

\begin{proposition}\label{propCI}
  Let $\chi_1$ and $\chi_2$ be non-trivial characters of $\Fqt$, and
  let $\mcF_1$, $\mcK$ and $\mcG$ be the constructible sheaves defined
  in \eqref{defF1} and \eqref{defKummersheaf} and let $\mcG^{0}$ be
  the weight $0$ part of $\mcG$.
  \begin{enumerate}
  \item The sheaves $\mcG$ and $\mcG^{0}$ have generic rank $2$.
  \item The sheaf $\mcG^{0}$ is lisse on $W=\Aa^1-\{0,-1\}$ and
    geometrically irreducible.
  \end{enumerate}
\end{proposition}

If we grant this proposition let us show how to conclude the proof of
Theorem~\ref{th-ci}. By Lemma \ref{lm-devisse}, we have
$$
p^{-1/2}S(\chi_1,\chi_2)=\sum_{x\in\Fp}t_{\mcF_1}(x)t_\mcG(x)
=\sum_{x\in\Fp}t_{\mcF_1}(x)t_{\mcG^{0}}(x)+O(p^{1/2})
$$
where the implicit constant is absolute. Since $\mcG^{0}$ is
geometrically irreducible of rank $2$ on~$W$ and $\mcF_1$ has rank
$1$, the sheaf $\mcG^{0}$ cannot not be geometrically isomorphic to
the dual of $\mcF_1$, and hence
$$
p^{-1/2}S(\chi_1,\chi_2)=O(p^{1/2})
$$
where the implied constant is absolute, by the Riemann Hypothesis and
the conductor bound $\cond(\mcG^0)\ll 1$. This finishes the proof of
Theorem~\ref{th-ci}.

For the proof of (2), we recall a very useful diophantine criterion
for irreducibility of Katz (see~\cite[Lemma 7.0.3]{katz-rls}).

\begin{lemma}[Irreducibility criterion]\label{lm-irred-crit}
  Let $\Fq$ be a finite field of characteristic $p$, let $\ell\not=p$
  be a prime number and let $\sheaf{F}$ be an $\ell$-adic
  constructible sheaf on $\Aa^1_{\Fp}$ which is mixed of weights
  $\leq 0$. Then we have
  \begin{equation}\label{eq-irred-hyp}
    \frac{1}{q^{\nu}} \sum_{x\in
      \Fqn}|\frfn{\sheaf{F}}(x,q^{\nu})|^2 =1+O(q^{-\nu/2})
  \end{equation}
  for $\nu\geq 1$, if and only if the middle-extension part of weight
  $0$ of $\sheaf{F}$ is geometrically irreducible, i.e., if and only if,
  for any dense open subset $U$ where $\sheaf{F}$ is lisse, the
  restriction of the weight $0$ part of $\sheaf{F}$ to $U\times\bFq$
  corresponds to an irreducible representation of the geometric
  fundamental group of $U$.
\end{lemma}

\begin{proof}
  For $\nu\geq 1$ fixed, let
$$
\frfn{\sheaf{F}}(x,q^{\nu})=\frfn{\sheaf{F}^{mid}}(x,q^{\nu})+t_1(x)+t_2(x)
$$
for $x\in\Fqn$ be the decomposition of Lemma~\ref{lm-devisse} (applied
to $\Fqn$). We wish to prove that $\sheaf{F}^{mid}$ is geometrically
irreducible. From the properties of $t_1$ and $t_2$, we see that
$$
\frac{1}{q^{\nu}} \sum_{x\in \Fqn}|\frfn{\sheaf{F}}(x,q^{\nu})|^2=
\frac{1}{q^{\nu}} \sum_{x\in
  \Fqn}|\frfn{\sheaf{F}^{mid}}(x,q^{\nu})|^2+O(q^{-\nu})
$$
for $\nu\geq 1$. Now let $U$ be a dense open subset of $\Aa^1$ where
$\sheaf{F}^{mid}$ is lisse. Then we have
$$
\frac{1}{q^{\nu}} \sum_{x\in
  U(\Fqn)}|\frfn{\sheaf{F}^{mid}}(x,q^{\nu})|^2= \frac{1}{q^{\nu}}
\sum_{x\in \Fqn}|\frfn{\sheaf{F}^{mid}}(x,q^{\nu})|^2+O(q^{-\nu}),
$$
for $\nu\geq 1$, since the complement is finite. Therefore, we
have~(\ref{eq-irred-hyp}) if and only if
$$
\frac{1}{q^{\nu}} \sum_{x\in
  U(\Fqn)}|\frfn{\sheaf{F}^{mid}}(x,q^{\nu})|^2=1+O(q^{-\nu/2})
$$
for $\nu\geq 1$. But by~\cite[Lemma 7.0.3]{katz-rls} applied to the
lisse sheaf $\sheaf{F}^{mid}$ on $U$, which is pure of weight $0$,
this last condition holds if and only if $\sheaf{F}^{mid}$ is
geometrically irreducible on $U$.
\end{proof}
 
\begin{proof}[Proof of Proposition \ref{propCI}]
  We begin by checking the generic rank of~$\mcG$. The fiber of
  $\sheaf{G}$ over $x\in \bFq$ is
  $$
  H^1_c(\Aa^1\times\bFq,\sheaf{L}_{\chi_1(Y(Y+1))}
  \otimes\sheaf{L}_{\chi_2(xY-1)})(1/2).
  $$
  \par
  By the Euler-Poincar\'e formula (see~(\ref{eq-euler-poincare-line})),
  its dimension is
  $$
  \dim H^1_c(\Aa^1\times\bFq,\sheaf{L}_{\chi_1(Y(Y+1))}
  \otimes\sheaf{L}_{\chi_2(xY-1)})=
  -1+3=2
  $$
  if $x\not=-1$ (so that the sheaf is ramified at the three points
  $y=0$, $-1$ and $1/x$). Hence the generic rank of~$\mcG$ is $2$.
  \par
  We next apply the irreducibility criterion to $\mcG$, which is mixed
  of weights $\leq 0$, to prove that the part of weight~$0$ is
  geometrically irreducible on any dense open set where it is lisse.
\par
For $\nu\geq 1$ and $i=1,2$, we denote by $\chi_{i,\nu}$ the extension
$\chi_i\circ N_{\Fqn/\Fq}$ of $\chi_i$ to $\Fqn$, we have
\begin{align*}
\frac{1}{q^{\nu}}\sum_{x\in \Fqn}|\frfn{\sheaf{G}}(x,q^{\nu})|^2&=
\frac{1}{q^{2\nu}}
\sum_{x\in\Fqn}
\Bigl|
\sum_{y\in \Fqn}\frfn{\sheaf{F}_1}(x,q^{\nu})
\chi_{2,\nu}(xy-1)
\Bigr|^2\\
&=\frac{1}{q^{2\nu}}
\sum_{y_1,y_2\in \Fqn}
\frfn{\sheaf{F}_1}(y_1,q^{\nu})
\overline{\frfn{\sheaf{F}_1}(y_2,q^{\nu})}
\sum_{x\in\Fqn}
\chi_{2,\nu}(xy_1-1)\overline{\chi_{2,\nu}(xy_2-1)}.
\end{align*}
\par
The contribution of the diagonal terms  $y_1=y_2$ to this sum is
\begin{equation}\label{F1geomirred}
	\frac{q^{\nu}-1}{q^{2\nu}} \sum_{y\in \Fqn}
|\frfn{\sheaf{F}_1}(y,q^{\nu})|^2= \frac{1}{q^{\nu}} \sum_{y\in \Fqn}
|\frfn{\sheaf{F}_1}(y,q^{\nu})|^2+O(q^{-\nu})=1+O(q^{-\nu})
\end{equation}
since $\frfn{\sheaf{F}_1}(y,q^{\nu})=\chi_{1,\nu}(y(y+1))$.
 \par
If $y_1\not=y_2$, the map 
$$
x\mapsto \frac{xy_1-1}{xy_2-1}
$$
is a bijection on $\Pp^1(\Fqn)$. Hence, in that case, we have
$$
\sum_{x\in\Fqn}
\chi_{2,\nu}(xy_1-1)\overline{\chi_{2,\nu}(xy_2-1)}=
-\chi_{2,\nu}(y_1)\overline{\chi_{2,\nu}(y_2)}
$$
(we write it in this way to incorporate the case $y_2=0$, in which
case the map is a bijection of $\Fqn$, while otherwise the sum over
$x\in\Fqn$ misses the point $y_1/y_2$.)
\par
Thus we get an off-diagonal contribution equal to
$$
-\frac{1}{q^{2\nu}}
\sum_{\stacksum{y_1,y_2\in \Fqn}{y_1\not=y_2}}
\frfn{\sheaf{F}_1}(y_1,q^{\nu})
\overline{\frfn{\sheaf{F}_1}(y_2,q^{\nu})}
\chi_{2,\nu}(y_1)\overline{\chi_{2,\nu}(y_2)}.
$$
\par 
Inserting the diagonal in this sum, we find that it is equal to
$$
-\frac{1}{q^{2\nu}}
\Bigl(
\Bigl|\sum_{y\in \Fqn}\frfn{\sheaf{F}_1}(y,q^{\nu})\chi_{2,\nu}(y)
\Bigr|^2-
\sum_{y\in\Fqn^{\times}}
|\frfn{\sheaf{F}_1}(y,q^{\nu})|^2
\Bigr).
$$
\par
Since $\sheaf{F}_1$ is geometrically irreducible but not geometrically
isomorphic to $\sheaf{L}_{\chi_2}$ (indeed $\mcF_1$ is ramified at
$-1$ while $\mcL_{\chi_2}$ is lisse there), by the Riemann Hypothesis
(in that case, due to A.Weil), we have
\begin{equation}\label{offdiag}
\Bigl|\sum_{y\in \Fqn}\frfn{\sheaf{F}_1}(y,q^{\nu})\chi_{2,\nu}(y)
\Bigr|^2=O(q^{\nu}),	
\end{equation}
while the bound
$$
\sum_{y\in\Fqn^{\times}} |\frfn{\sheaf{F}_1}(y,q^{\nu})|^2=O(q^{\nu})
$$
is immediate. Hence the off-diagonal contribution is $O(q^{-\nu})$,
and the irreducibility criterion does apply.
\par
Thus~$\mcG^0$ is geometrically irreducible on any open set where it is
lisse. We will now prove that~$\mcG$ is lisse and pure of weight~$0$
on~$W$. It then follows that~$\mcG=\mcG^0$ on~$W$, which will conclude
the proof of the proposition.

We begin by checking that $\mcG$ is lisse on $W=\Aa^1- \{0,-1\}$ using
Deligne's semicontinuity theorem (\cite[Cor. 2.1.2]{deligne-l}). We
denote by~$p_1$ and~$p_2$ the projections $(x,y)\mapsto x$ and
$(x,y)\mapsto y$ from $\Aa^2$ to~$\Aa^1$. We also denote by
$\widetilde{p}_1\colon \Aa^1\times\Pp^1\to \Aa^1$ the first
projection.  This is a smooth and proper morphism of relative
dimension~$1$.
Let~$\mcH=\mcL_{\chi_2(XY-1)}\otimes p_2^*\mcL_{\chi_1(Y(Y+1))}$ so
that~$\mcG=R^1p_{1,!}\mcH(1/2)$. Note that~$\mcH$ is lisse on
$U=\Aa^2-D$ where~$D$ is the divisor
$$
D=\{XY=1\}\cup (\Aa^1\times \{0\})\cup (\Aa^1\times \{-1\}).
$$
\par
We denote by $\widetilde{\mcH}$ the sheaf on~$\Aa^1\times\Pp^1$ which
is the extension by zero of~$\widetilde{\mcH}$ from~$\Aa^1\times\Aa^1$
to~$\Aa^1\times \Pp^1$. By definition, we have
$\mcG=R^1\widetilde{p}_{1,*}\widetilde{\mcH}(1/2)$.

Let~$\widetilde{D}$ be the complement in~$\Aa^1\times\Pp^1$ of the
open set~$U$. This is the union of~$D$ and of the line
$\Aa^1\times\{\infty\}$. 

Let~$X=\widetilde{p}_1^{-1}(W)$. By restriction, the
morphism~$\widetilde{p}_1$ defines a proper smooth morphism $X\to W$
of relative dimension~$1$. The intersection~$\widetilde{D}\cap X$ is a
divisor in~$X$ that is flat and finite (of degree~$4$) over~$W$. The
sheaf~$\widetilde{\mcH}$ is lisse on the complement
of~$\widetilde{D}\cap X$ in~$\Aa^1\times\Pp^1$.

Let~$x\in W$. The fiber $C_x$ of~$\widetilde{p}_1$ over~$x$ is
identified with~$\Pp^1$, and the restriction of~$\widetilde{\mcH}$
to~$C_x$ is identified with a lisse sheaf on the dense open set
$$
U_x=\Aa^1-\{0,-1,1/x,\infty\}\subset \Pp^1.
$$
The restriction of the sheaf~$\widetilde{\mcH}$ to~$C_x$ is (at most)
tamely ramified everywhere, hence the function~$\varphi$
of~\cite[Th. 2.1.1]{deligne-l} is constant equal to~$0$ on points
of~$W$. Corollary 2.1.2 of loc. cit. then implies that~$\mcG$ is lisse
on~$W$, as claimed.
\par
We finally prove that~$\mcG$ is pure of weight~$1$ on~$W$. We
apply~\cite[Lemma 4.22 (b)]{KMS} to the
morphism~$\widetilde{p}_1\colon X\to W$ and to the
sheaf~$\widetilde{\mcH}$ on~$X$.  For any~$x\in \Pp^1$, the
sheaf~$\widetilde{\mcH}_x$, after pullback to~$C_x=\{x\}\times\Pp^1$,
has neither punctual section nor trivial subrepresentation (as lisse
sheaf on~$U_x$).  Thus the assumptions of loc. cit. are satisfied.  It
follows that for any~$x\in W$, the part of weight~$<1$ of the stalk
at~$x$ of~$\mcG^{mid}$ is isomorphic to
$$
\bigoplus_{y\in C_x-U_x} (\widetilde{\mcH}_x)_{\bar{\eta}}^{I_y}/
(\widetilde{\mcH}_x)_{\bar{y}}.
$$
But already $(\widetilde{\mcH}_x)_{\bar{\eta}}^{I_y}=0$ at all
singular points $y\in \{0,-1,1/x,\infty\}$, so this direct sum
vanishes.
\end{proof}

\begin{remark}
  The irreducibility criterion applies more generally to show that
  $T_\mcK(\mcF)^{mid}$ is geometrically irreducible as long as $\mcF$
  is a geometrically irreducible middle-extension sheaf, pure of
  weight $0$, which is not geometrically isomorphic to
  $\mcL_{\chi_2}$. Indeed, under these assumptions, the irreducibility
  criterion shows that \eqref{F1geomirred} holds with $\mcF_1$
  replaced by $\mcF$, while \eqref{offdiag} follows from the Riemann
  Hypothesis of Deligne.
\end{remark}

% \subsection{Application to the $q$-van der Corput method}

%%%%%%%%%%%%%%%%%%%%%%

\section{Setting up the proof}
\label{sec-setup}

To clarify the proof of Theorems~\ref{th-conductor}
and~\ref{th-betti}, and in view of further generalizations, we
introduce the following definition:

\begin{definition}[Continuity]
  (1) Let
$$
i\,:\, (f,\sheaf{F})\mapsto i(f,\sheaf{F})
$$
be any real-valued map taking a pair $(f,\sheaf{F})$ as input, where
$f$ is a non-constant rational function in $\Fq(X,Y)$ for some finite
field $\Fq$ and $\sheaf{F}$ is a middle-extension $\ell$-adic sheaf on
the affine line over $\Fq$.  Then we say that $i$ is \emph{continuous}
if there exists an integer $C\geq 1$ such that
$$
|i(f,\sheaf{F})|\leq (2\cond(f)\cond(\sheaf{F}))^C
$$
for all pairs $(f,\sheaf{F})$ as above such that
$\cond(f)<p$.\footnote{\ This restriction on~$\cond(\sheaf{F})$ may
  seem artificial, and it is possible that it would not be needed for
  our results. But it has no influence on the applications.}
\par
(2) Similarly, if
$$
j\,:\, f\mapsto j(f)\ (\text{resp. } k\,:\, \sheaf{F}\mapsto k(\sheaf{F}))
$$
are real-valued maps taking as input a non-constant rational function
$f\in \Fq(X,Y)$ for some finite field $\Fq$ (resp. a
middle-extension $\ell$-adic sheaf $\sheaf{F}$ on the affine line over
$\Fq$), then we say that $j$ (resp. $k$) is \emph{continuous} if and
only if there exists an integer $C\geq 1$ such that
$$
|j(f)|\leq (2\cond(f))^C (\text{resp. } |k(\sheaf{F})|\leq
(2\cond(\sheaf{F}))^C),
$$
for all $f$ with $\cond(f)<p$ (resp. all middle-extension sheaves
$\sheaf{F}$).
\par
\end{definition}

\begin{remark}
Some of our arguments are easier to follow and check if one uses a
weaker definition of continuity, where one only asks that
$$
|i(f,\sheaf{F})|\leq \Psi(\cond(f),\cond(\sheaf{F}))
$$
for some function $\Psi$ taking positive integral values. For some
basic applications, such a statement is also sufficient, and the
reader might wish to consider this as the notion of continuity in a
first reading.
\end{remark}

\begin{example}
For instance, Theorem~\ref{th-conductor} asserts that the maps
$$
(f,\sheaf{F})\mapsto \cond(T^i_{\sheaf{K}}(\sheaf{F}))
$$
are continuous, and Theorem~\ref{th-betti} that the maps
$$
(f,\sheaf{F})\mapsto \dim
H^i_c(\Aa^2\times\bFq,p_2^*\sheaf{F}\otimes\sheaf{K})
$$
are continuous. Lemma~\ref{lm-cont-hic} proves that the functions
$$
\sheaf{F}\mapsto \dim H^i_c(\Aa^1\times\bFq,\sheaf{F})
$$
are continuous.
\end{example}

Clearly, if we fix one argument of a continuous map $i(f,\sheaf{F})$
and let the other vary, this gives a continuous map of this second
argument. Also, a sum $i_1+i_2$ of continuous functions is also
continuous, as well as a product $i_1i_2$.
\par
For simplicity, we denote
\begin{align*}
c_i(f,\sheaf{F})&=\cond(T^i_{\sheaf{K}}(\sheaf{F})),\quad\quad 0\leq i\leq 2\\
h^j(f,\sheaf{F})&=\dim
H^j_c(\Aa^2\times\bFq,p_2^*\sheaf{F}\otimes
\sheaf{L}_{\psi(f)}),\quad\quad 0\leq j\leq 4\\
%m(f,\sheaf{F})&=\rank(T^1_{\sheaf{K}}(\sheaf{F}))+
%n(T^1_{\sheaf{K}}(\sheaf{F}))+\pct(T^1_{\sheaf{K}}(\sheaf{F})).
m(f,\sheaf{F})&=\rank(T^1_{\sheaf{K}}(\sheaf{F}))+\pct(T^1_{\sheaf{K}}(\sheaf{F})).
\end{align*}
\par
The proof of Theorems~\ref{th-conductor} and~\ref{th-betti} will be
based on the following steps:

\begin{proposition}\label{pr-steps}
The following assertions are true:
\begin{enumerate}
	\item \label{un} The map $(f,\sheaf{F})\mapsto c_0(f,\sheaf{F})$ is continuous.

\item \label{de}  For $0\leq j\leq 4$, the map
$$
f\mapsto h^j(f,\bQl)=\dim
H^j_c(\Aa^2\times\bFq,\sheaf{L}_{\psi(f)})
$$ 
is continuous.
\item 
\begin{enumerate}

\item\label{trb} If $f\mapsto h^2(f,\bQl)$ is continuous, then $f\mapsto c_2(f,\bQl)$ is
continuous;
\item   \label{tr} if $(f,\sheaf{F})\mapsto h^2(f,\mcF)$ is continuous, then
$(f,\sheaf{F})\mapsto c_2(f,\sheaf{F})$ is continuous.
\end{enumerate}
\item\label{qu} If $f\mapsto c_1(f,\bQl)$ and $f\mapsto c_2(f,\bQl)$ are
continuous, then $(f,\sheaf{F})\mapsto h^2(f,\sheaf{F})$ is conti\-nuous.
\item  
\begin{enumerate}
\item\label{cib}  If $f\mapsto c_2(f,\bQl)$ is continuous, then 
$f\mapsto m(f,\bQl)$ is continuous;
\item \label{ci} if $(f,\sheaf{F})\mapsto c_2(f,\sheaf{F})$ is continuous, then  $(f,\sheaf{F})\mapsto m(f,\sheaf{F})$ is continuous.

\end{enumerate}
\item 
\begin{enumerate}
\item\label{sib} If $f\mapsto m(f,\bQl)$ and $f\mapsto h^2(f,\bQl)$ are both
continuous, then  $f\mapsto c_1(f,\bQl)$ is continuous;
\item\label{si} if $(f,\sheaf{F})\mapsto m(f,\sheaf{F})$ and $(f,\sheaf{F})\mapsto h^2(f,\sheaf{F})$ are both
continuous, then  $(f,\sheaf{F})\mapsto c_1(f,\sheaf{F})$ is continuous.

\end{enumerate}
\item\label{se} If $(f,\sheaf{F})\mapsto c_i(f,\sheaf{F})$ is continuous for $0\leq i\leq 2$,
then $(f,\sheaf{F})\mapsto h^j(f,\sheaf{F})$ is continuous for all $0\leq j\leq 4$.
\end{enumerate}

\end{proposition}

We now explain how to deduce Theorems~\ref{th-conductor}
and~\ref{th-betti} from this proposition. Since this may also look
like spaghetti-mathematics, the reader may also wish to go straight to
Sections~\ref{sec-spectral} and~\ref{sec-polymath} (possibly in the
opposite order) which together give an account of the proof for the
special case of the Fourier transform (and discuss another example
arising in the \textsc{Polymath8} project), in which case the flow of
the proof is much easier to follow.
\par
\medskip
\par
First of all, $c_0$ is continuous by \eqref{un}, so we must show that $c_1$,
$c_2$ and the $h^j$ are continuous.
\par
\textbf{Step 1.} Using \eqref{de}, we can apply \eqref{trb} and deduce
that $f\mapsto c_2(f,\bQl)$ is continuous. By \eqref{cib}, it follows
that $m(f,\bQl)$ is continuous. Combining this with \eqref{sib} and
\eqref{de} again, we deduce that $f\mapsto c_1(f,\bQl)$ is continuous.
\par
At this point, we have proved both theorems in the special case when
$\sheaf{F}=\bQl$ is the trivial sheaf.
\par
\textbf{Step 2.} From \eqref{qu} and Step 1, we see that
$(f,\sheaf{F})\mapsto h^2(f,\sheaf{F})$ is continuous. This fact
combined with \eqref{tr} shows that
$(f,\sheaf{F})\mapsto c_2(f,\sheaf{F})$ is continuous. In turn,
\eqref{ci} then proves that $(f,\sheaf{F})\mapsto m(f,\sheaf{F})$ is
continuous, and finally \eqref{si} allows us to conclude that
$(f,\sheaf{F})\mapsto c_1(f,\sheaf{F})$ is continuous.
\par
At this point we have proved Theorem~\ref{th-conductor} (and the
continuity of $(f,\sheaf{F})\mapsto h^2(f,\sheaf{F})$); by \eqref{se},
we deduce that all $h^j$ are continuous.

\begin{remark}
  (1) We will in fact establish \eqref{tr} and \eqref{trb} directly by
  proving a direct relation between $c_2(f,\sheaf{F})$ and
  $h^2(f,\sheaf{F})$ for $\mcF=\bQl$ or in general, and similarly for
  \eqref{ci} and \eqref{cib}, \eqref{si} and \eqref{sib}.
\par
(2) The most crucial points in Proposition~\ref{pr-steps} are
\begin{itemize}
\item[--] \eqref{de}, which gives the starting point of the argument
  for the trivial sheaf, and which comes from the bounds for Betti
  numbers of Bombieri, Adolphson-Sperber and Katz.
\item [--] \eqref{qu}, which allows us to pass from properties known
  for the trivial sheaf only, to properties of all sheaves.
\end{itemize} 
%On
%the other hand, what turns out to be a bit more involved (though relatively
%elementary) is the proof of \eqref{ci} and \eqref{cib}, which amounts to
%controlling all invariants defining the conductor of
%$T^1_{\sheaf{K}}(\sheaf{F})$ except for the sum of Swan
%conductors. This part was essentially swept under the rug in
%Section~\ref{sec-motivation}, which explains partly why the proof of
%Theorem~\ref{th-conductor} is quite a bit longer than that motivating
%sketch might suggest.
\par
(3) It is only in the proof of \eqref{ci} and \eqref{cib} that we will use the
restriction that continuity applies to $f$ with $\cond(f)<p$.
\end{remark}

\section{Spectral sequence argument}\label{sec-spectral}

We state here the few simple facts about spectral sequences that we
require. We first recall the basic formalism, referring
to~\cite[Appendix. B]{milne} for a survey and~\cite[Ch. 10]{rotman}
for details.
\par
Let $k$ be a fixed field. A converging (first quadrant) spectral
sequence
$$
E_2^{p,q}\Rightarrow E^n,
$$
of $k$-vector spaces involves (1) vector spaces $E_2^{p,q}$ defined
for $p$, $q\geq 0$; (2) vector spaces $E^n$ defined for $n\geq 0$; (3)
linear maps
\begin{equation*}%\label{eq-d2}
d_2^{p,q}\,:\, E_2^{p,q}\lra E_2^{p+2,q-1},
\end{equation*}
(called differentials)\footnote{\ Note that these differentials show
  that $p$ and $q$ do not play symmetric roles.} for all $p$ and $q$
(with the convention $E_2^{p,q}=0$ if $p$ or $q$ is negative), such
that
$$
d_2^{p,q}\circ d_2^{p-2,q+1}=0.
$$

\begin{remark}
The use of the indices $p$ and $q$ for the spectral sequence is almost
universal, although it clashes with the usual convention that $p$ is a
prime and $q$ a power of $p$.  We will use $i$ and $j$ instead of $p$
and $q$ when both notation are involved, although the difference in
context should avoid confusion.
\end{remark}

One defines
\begin{equation}\label{eq-e3}
E_3^{p,q}=\ker d_2^{p,q}/\Imag d_2^{p+2,q-1},
\end{equation}
and one shows that there are linear maps
\begin{equation}\label{eq-d3}
d_3^{p,q}\,:\, E_3^{p,q}\lra E_3^{p+3,q-2},
\end{equation}
such that $d_3^{p,q}\circ d_3^{p-3,q+2}=0$. This process is then
suitably iterated to obtain $E_j^{p,q}$ for all $j\geq 2$, and
differentials
$$
d_j^{p,q}\,:\, E_j^{p,q}\lra E_{j}^{p+j,q-j+1}
$$
(with composites vanishing).
\par
One says that \emph{the spectral sequence degenerates at the
  $E_j$-level} (where $j=2$ or $3$) if $d_i^{p,q}=0$ for all $p$,
$q\geq 0$ and $i\geq j$. When this is the case, the formalism gives
(among other things) the following relation between the $E_{j}^{p,q}$
and the spaces $E^n$: we have for all $n\geq 0$, a (non-canonical) isomorphism
\begin{equation}\label{eq-convergence}
E^n\simeq \bigoplus_{p=0}^n E_j^{p,n-p},
\end{equation}
of $k$-vector spaces. (There is often more structure involved, but
this will suffice for us.)
\par
Furthermore, whether the spectral sequence degenerates at the $E_2$ or
$E_3$ level or not, there is an exact sequence
\begin{equation}\label{eq-exact-seq}
0\ra E_2^{1,0}\lra E^1\lra E_{2}^{0,1}\lra E_2^{2,0}.
\end{equation}
\par
All these facts are stated in~\cite[p. 307--309]{milne}.  The next
proposition then summarizes all results we will need from spectral
sequences:

\begin{proposition}\label{pr-prop-seq}
  Let $k$ be a field and let
$$
E_2^{p,q}\Rightarrow E^n
$$
be a converging spectral sequence as above. Assume that $E_2^{p,q}=0$
unless $0\leq p\leq 2$ and $0\leq q\leq 2$.
\par
\emph{(1)} The spectral sequence degenerates at the $E_3$-level and we
have
\begin{equation}\label{eq-e211}
  E^2\simeq E_3^{0,2}\oplus E_2^{1,1}\oplus E_3^{2,0}.
\end{equation}
\par
\emph{(2)} We have
$$
\dim E^n\leq \sum_{p=0}^n \dim E_2^{p,n-p},
$$
and
$$
\dim E_2^{0,2}\leq \dim E^2+\dim E_2^{2,1}.
$$
\par
\emph{(3)}  Assume in addition that $E_2^{p,q}=0$ if $q=0$. We have then
% $$
% \dim E_2^{0,2}\leq \dim E^2+\dim E_2^{2,1},
% $$
% and
$E_2^{0,1}\simeq E^1$.
\end{proposition}

\begin{proof}
  (1) From~(\ref{eq-e3}), we see that $E_3^{p,q}=0$ unless $0\leq
  p,q\leq 2$ since it is a quotient of a subspace of $E_2^{p,q}$. But
  then~(\ref{eq-d3}) shows that, for any $p$, $q$, either the source
  of the target of $d_3^{p,q}$ is zero. In fact, for all $i\geq 3$,
  either the target or the source of $d_i^{p,q}$ vanishes, and
  therefore the spectral sequence degenerates at that level.
\par
By~(\ref{eq-convergence}) we deduce that
$$
E^2\simeq E_3^{0,2}\oplus E_3^{1,1}\oplus E_3^{0,2},
$$
but
$$
E_3^{1,1}=\ker d_2^{1,1}/\Imag d_2^{-1,2},
$$
and since $d_2^{1,1}$ and $d_2^{-1,2}$ are both zero (the target of
the first and the source of the second are zero), we have
$E_3^{1,1}=E_2^{1,1}$, hence~(\ref{eq-e211}).
% \par
% The assumption and~(\ref{eq-d2}) show that the only possible
% non-zero differentials at level $2$ are
% $$
% d_2^{0,1}\,:\, E_2^{0,1}\lra E_2^{2,0},\quad\quad
% d_2^{0,2}\,:\, E_2^{0,2}\lra E_2^{2,1}.
% $$
\par
(2) By (1) and~(\ref{eq-convergence}), we have
$$
E^n\simeq E_3^{n,0}\oplus E_3^{n-1,1}\oplus \cdots\oplus E_3^{0,n}.
$$
\par
Since
$$
\dim E_3^{p,q}\leq \dim E_2^{p,q},
$$
for all $p$ and $q$, by~(\ref{eq-e3}), we obtain
$$
\dim E^n=\sum_{p=0}^n \dim E_3^{p,q}\leq \sum_{p=0}^n
\dim E_2^{p,q}.
$$
\par
Similarly, we note that
$$
E_3^{0,2}=\ker d_2^{0,2}/\Imag d_2^{-2,1}=\ker d_2^{0,2},
$$
and hence we have a short exact sequence
$$
0\lra E_3^{0,2}\lra E_2^{0,2}\fleche{d_2^{0,2}} E_2^{2,1}
$$
which implies that
$$
\dim E_2^{0,2}\leq \dim E_3^{0,2}+ \dim E_2^{2,1}.
$$
\par
% From~(\ref{eq-e3}), we have
% $$
% E_3^{2,0}=\ker d_2^{2,0}/\Imag d_2^{0,1}=\ker d_2^{2,0}.
% $$
% since the assumptions and~(\ref{eq-d2}) show that the only possible
% non-zero differential at level $2$ is
% $$
% d_2^{0,2}\,:\, E_2^{0,2}\lra E_2^{2,1}.
% $$
% \par
From the degeneracy at the $E_3$-level, we then get
$$
\dim E_3^{0,2}\leq \dim E^2,
$$
hence the bound for $E_2^{0,2}$.
\par
(3) The exact sequence~(\ref{eq-exact-seq}), under the assumptions
that $E_2^{p,0}=0$, becomes
$$
0\lra E^1\lra E_2^{0,1}\lra 0,
$$
hence the result.
\end{proof}

The spectral sequences we use are given by the following lemma:

\begin{lemma}\label{lm-leray-seq}
Let $\Fq$ be a finite field of characteristic $p$, $\ell\not=p$ a
prime number. Let $f\in\Fq(X,Y)$ be a rational function, and denote
$$
\sheaf{K}=\sheaf{L}_{\psi(f(X,Y))},
$$
where $\psi$ is a non-trivial additive $\ell$-adic character. Denote
$f^*(X,Y)=f(Y,X)\in \Fq$, and
$$
\sheaf{K}^*=\sheaf{L}_{\psi(f^*)}.
$$
\par
Let
$\sheaf{F}$ be a constructible $\ell$-adic sheaf on $\Aa^1_{\Fq}$.
\par
\emph{(1)} For any dense open subsets $U$, $V$ of $\Aa^1_{\Fq}$,
with $p_1$, $p_2$ denoting the projection maps $U\times V\lra U$ and
$U\times V\lra V$, respectively, there exist converging spectral
sequences
\begin{align*}
  E_2^{i,j}=H^i_c(\bar{U},T^j_{\sheaf{K}}(\sheaf{F}))&\Rightarrow
  H^{i+j}_c(\bar{U}\times\Aa^1,p_2^*\sheaf{F}\otimes\sheaf{K}),\\
  E_2^{i,j}=H^i_c(\bar{V},\sheaf{F}\otimes
  T^j_{\sheaf{K}^*}(\bQl))&\Rightarrow
  H^{i+j}_c(\Aa^1\times\bar{V},p_2^*\sheaf{F}\otimes\sheaf{K})
\end{align*}
of $\bQl$-vector spaces
\par
\emph{(2)} These two spectral sequences satisfy $E_2^{i,j}=0$ unless
$0\leq i\leq 2$ and $1\leq j\leq 2$.
\end{lemma}

\begin{proof}
  (1) The first spectral sequence is the Leray spectral sequence of
  the first projection map $p_1\,:\, U\times\Aa^1\ra U$ and of the
  sheaf $p_2^*\sheaf{F}\otimes\sheaf{K}$ (see, e.g.,~\cite[Th. 7.4.4
  (ii)]{leifu} or~\cite[Th. VI.3.2 (c)]{milne}.)
\par
The second spectral sequence arises from the Leray spectral sequence
of the second projection $p_2\,:\, \Aa^1\times V\ra V$ and of the
sheaf $p_2^*\sheaf{F}\otimes\sheaf{K}$, namely
$$
E_2^{i,j}=H^i_c(\bar{V},R^jp_{2,!}(p_2^*\sheaf{F}\otimes\sheaf{K}))
\Rightarrow
H^{i+j}_c(\Aa^1\times \bar{V},p_2^*\sheaf{F}\otimes\sheaf{K})
$$
together with the facts that
$$
R^jp_{2,!}(p_2^*\sheaf{F}\otimes\sheaf{K})=\sheaf{F}\otimes 
R^jp_{2,!}(\sheaf{L}_{\psi(f(X,Y))})
$$
by the projection formula (see, e.g.,~\cite[Th. 7.4.7]{leifu}), and
that we can identify $R^jp_{2,!}(\sheaf{L}_{\psi(f(X,Y))})$ with
$T^j_{\sheaf{K}^*}(\bQl)$ (restricted to $V$).
\par
(2) The fact that $E_2^{i,j}=0$ unless $0\leq i,j\leq 2$ is immediate
from (1) and from the vanishing of cohomology of curves (resp. of
higher-direct image sheaves for maps with curves as fibers) in
Proposition~\ref{pr-etale}, (1): the former constrains $i$ to be
between $0$ and $2$, and the second constrains similarly $j$.
\par
For the vanishing when $j=0$, we note that the stalk at $x\in
\Aa^1(\bFq)$ of $R^0p_{1,!}(p_2^*\sheaf{F}\otimes\sheaf{K})$ is, by
the proper base change theorem, equal to
$$
H^0_c(\Aa^1\times\bFq,\sheaf{F}\otimes\sheaf{L}_{\psi(f(x,Y))})=0 
$$
by Lemma~\ref{lm-tensor}. Similarly, the stalk of
$R^0p_{2,!}(\sheaf{K})$ at $y$ is
$$
H^0_c(\Aa^1\times\bFq,\sheaf{L}_{\psi(f(X,y))})=0,
$$ 
and these facts show that $E_2^{i,0}=0$ for all $i$ in both spectral
sequences.
\end{proof}

\section{Beginning of the proof}\label{sec-begin}

We will now begin the proof of Proposition~\ref{pr-steps}. As a
warm-up, the reader may wish to have a look at Section
\ref{sec-polymath} where we discuss the simpler case of the Fourier
transform (where $f(X,Y)=XY$) and a closely related case appearing in
the \textsc{Polymath8} project.
\par
We first deal with parts \eqref{un} and \eqref{de} of
Proposition~\ref{pr-steps}.
\par
(1) We claim that $T^0_{\sheaf{K}}(\sheaf{F})=0$ for all $f$ and
$\sheaf{F}$. Indeed, by the proper base change theorem
(Proposition~\ref{pr-etale}, (4)), the stalk of
$R^0p_{1,!}(p_2^*\sheaf{F}\otimes\sheaf{K})$ over $x\in\Aa^1(\bFq)$ is
$$
H^0_c(\Aa^1\times\bFq,\sheaf{F}\otimes\sheaf{L}_{\psi(f(x,Y))})=0
$$
by Lemma~\ref{lm-tensor}.
\par
(2) By the bounds of Bombieri, Adolphson-Sperber and Katz (see,
e.g.,~\cite[Th. 12]{katz-betti}), the sum of Betti numbers
$$
\sum_{i=0}^4\dim H^i_c(\Aa^2\times\bFq,\sheaf{L}_{\psi(f)})
$$
is bounded by $(1+\cond(f))^B$ for some absolute constant $B\geq 1$,
% depending only on the degree of the numerator and denominator of
% $f$,
which proves the continuity of $h^i(f,\bQl)$. Precisely, in order to
apply the result of Katz, one writes $f=f_1/f_2$ with $f_i\in\Fq[X,Y]$
and $f_1$ coprime to $f_2$, then one notes that if $U_2\subset \Aa^2$
is the open subset where the denominator $f_2$ is invertible, we have
$$
H^i_c(\Aa^2\times\bFq,\sheaf{L}_{\psi(f)})
=H^i_c(U_2\times\bFq,\sheaf{L}_{\psi(f)}) 
$$
by definition of cohomology with compact support. Define $Z\subset
\Aa^3$, where $\Aa^3$ has coordinates $(U,X,Y)$, to be the zero set of
the polynomial $Uf_2(X,Y)-1$. Then the morphism
$$
\alpha
\begin{cases}
  Z\lra U_2\\
  (u,x,y)\mapsto (x,y)
\end{cases}
$$
is an isomorphism such that $\alpha^*\sheaf{L}_{\psi(f)}$ is
isomorphic to the lisse sheaf $\sheaf{L}_\psi(\tilde{f})$ for the
polynomial $\tilde{f}=Uf_1(X,Y)\in \Fq[U,X,Y]$. Katz's theorem gives
precisely the upper-bound
$$
\sum_{i=0}^4\dim H^i_c(\bar{Z},\sheaf{L}_{\psi(\tilde{f})}) \leq
3\Bigl(1+1+\max(1+\deg f_1,1+\deg f_2)\Bigr)^{3+1},
$$
and hence the result.
\par
The other parts of the proof are more involved, and require the tools
of Section~\ref{sec-spectral}.  However, before going further we will
deal directly with the special case when $f\in\Fq(X)+\Fq(Y)$ (the
reader is invited to figure out the analogue of
Section~\ref{sec-motivation} in this case).
\subsection{Proof of Theorems~\ref{th-conductor} and~\ref{th-betti} in a factorable case}\label{sec-factorable}

So assume that
$$
f=f_1+f_2,
$$
with $f_1\in\Fq(X)$ and $f_2\in\Fq(Y)$. We have
$\sheaf{K}=p_1^*\sheaf{L}_1\otimes p_2^*\sheaf{L}_2$,
where $\sheaf{L}_i=\sheaf{L}_{\psi(f_i)}$, hence
$$
R^ip_{1,!}(p_2^*\sheaf{F}\otimes p_2^*\sheaf{L}_2 \otimes
p_1^*\sheaf{L}_1)\simeq \sheaf{L}_1\otimes
R^ip_{1,!}(p_2^*(\sheaf{F}\otimes\sheaf{L}_2)),
$$
for $0\leq i\leq 2$, by the projection formula (see,
e.g.,~\cite[Th. 7.4.7]{leifu}).
\par
But the sheaf $R^ip_{1,!}(p_2^*(\sheaf{F}\otimes\sheaf{L}_2))$ is the
constant sheaf associated to
$H^i_c(\Aa^1\times\bFq,\sheaf{F}\otimes\sheaf{L}_2)$: indeed,
applying~\cite[Arcata, IV, Th. 5.4]{deligne} to the cartesian diagram
$$
\begin{array}{ccc}
\Aa^1 & \stackrel{p_2}{\longleftarrow} & \Aa^2\\
s_2\downarrow & & \downarrow p_1\\
\spec\Fq & \stackrel{s_1}{\longleftarrow} & \Aa^1
\end{array}
$$
and the sheaf $\sheaf{F}\otimes\sheaf{L}_2$ on $\Aa^1$, we obtain
$$
s_1^*R^is_{2,!}(\sheaf{F}\otimes\sheaf{L}_2)\simeq
R^ip_{1,!}(p_2^*(\sheaf{F}\otimes\sheaf{L}_2)),
$$
and the left-hand side is a constant sheaf (since it is pulled-back
from $\Fq$) and has fiber $R^is_{2,!}(\sheaf{F}\otimes\sheaf{L}_2)=
H^i_c(\Aa^1\times\bFq,\sheaf{F}\otimes\sheaf{L}_2)$, by the definition
of cohomology with compact support and higher-direct images.
\par
Hence we have (see~(\ref{eq-subadd}))
$$
c_i(f_1+f_2,\sheaf{F})\leq (\dim
H^i_c(\Aa^1\times\bFq,\sheaf{F}\otimes\sheaf{L}_2))
\times \cond(\sheaf{L}_1),
$$
which is continuous as a function of $\cond(\sheaf{F})$ and $\cond(f)$
by Lemmas~\ref{lm-cont-hic} and~\ref{lm-tensor-cond}.
\par
This proves Theorem~\ref{th-conductor} in the special case
$f\in \Fq(X)+\Fq(Y)$, and Theorem~\ref{th-betti} follows either from
the argument in Section~\ref{sec-7} (which is general) or from an
application of the K\"unneth formula (see, e.g.,~\cite[Sommes Trig.,
(2.4)*]{deligne}) and of Lemma~\ref{lm-cont-hic}.

\begin{remark}\label{rem-factor}
  In particular, by the definition, of the conductor, we have
  established Proposition in the case $f\in\Fq(X)+\Fq(Y).$ In the
  sequel, we may (and will) assume that
  $$
  f\not\in\Fq(X)+\Fq(Y).
  $$
\end{remark}

\section{Proposition~\ref{pr-steps}: proof of \eqref{tr} and \eqref{trb}}

We prove \eqref{tr} and assume that the function $(f,\sheaf{F})\mapsto h^2(f,\sheaf{F})$ is
continuous, and our aim is to show that
$$(f,\sheaf{F})\mapsto c_2(f,\sheaf{F})=\cond(T^2_{\sheaf{K}}(\sheaf{F}))$$ is continuous. As pointed out in Remark \ref{rem-factor}, we can assume from now on that
$$f\notin \Fq(X)+\Fq(Y).$$
\begin{lemma}\label{lm-vanish}
  Assume $f\notin \Fq(X)+\Fq(Y)$. Then $T^2_{\sheaf{K}}(\sheaf{F})$
  vanishes generically.
\end{lemma}

\begin{proof}
Denote
$$
\sheaf{G}=T^2_{\sheaf{K}}(\sheaf{F})= R^2p_{1,!}(p_2^*\sheaf{F}\otimes
\sheaf{K}).
$$
\par
Let $(\sheaf{F}_i)$ be the (geometric) Jordan-H\"older factors of
$\sheaf{F}$; then the geometric Jordan-H\"older factors of
$p_2^*\sheaf{F}\otimes \sheaf{K}$ are the
$p_2^*\sheaf{F}_i\otimes\sheaf{K}$, so that we may assume that
$\sheaf{F}$ is geometrically irreducible.
\par
Let $\eta=\spec(\Fq(X))$ be the generic point of the affine line
$\Aa^1_{\Fq}$ (with coordinate $X$), let
$\bar{\eta}=\spec(\overline{\Fq(X)})$ be a geometric point above
$\eta$. By constructibility, the stalks of $\sheaf{G}$ vanish for all
$x$ in a dense open subset if and only if the stalk
$\sheaf{G}_{\bar{\eta}}$ is zero.
\par
By the proper base change theorem, we have
$$
\sheaf{G}_{\bar{\eta}}=H^2_c(\Aa^1\times\overline{\Fq(X)},
\sheaf{F}\otimes \sheaf{L}_{\psi(f_X(Y))}),
$$
where $f_X(Y)=f(X,Y)$. Assume this stalk is non-zero. Then, using the
coinvariant formula for the second cohomology group on a curve, it
follows that there exists an open subset $U$ of the affine line (with
coordinate $Y$) over $\overline{\Fq(X)}$ such that
$$
\sheaf{F}\simeq \sheaf{L}_{\psi(-f_X(Y))}
$$
as sheaves on $U\times \overline{\Ff_{q}(X)}$. Since they are
middle-extension sheaves, they are isomorphic as sheaves on the affine
line over $\overline{\Fq(X)}$.
\par
Note that $\sheaf{F}$ is pulled back from the affine line $\Aa^1$ over
$\Fq$ (still with coordinate $Y$), and so the classification of
Artin-Schreier sheaves shows that $f$ is, up to an additive
``constant'' in $\overline{\Fq(X)}$, an element in $\Fq(Y)$, i.e.,
we have
$$
f=f_1+f_2,
$$
with $f_1\in\Fq(X)$ and $f_2\in\Fq(Y)$.
\end{proof}

\begin{remark}
  One can also prove this lemma using more elementary arguments on
  rational functions, by looking at the vanishing at individual stalks
  and the classification of Artin-Schreier sheaves on $\Aa^1_{\Fq}$.
\end{remark}

Because of this lemma, the conductor of
$\sheaf{G}=T^2_{\sheaf{K}}(\sheaf{F})$ is equal to $\pct(\sheaf{G})$
(the generic rank is $0$, and thus the action of all inertia groups on
the generic fiber is trivial, which implies that $n(\sheaf{G})=0$ and
hence the Swan conductors also vanish.) Hence
$$
\cond(\sheaf{G})=\dim
H^0_c(\Aa^1\times\bFq,T^2_{\sheaf{K}}(\sheaf{F})). 
$$
\par
In the first spectral sequence of Lemma~\ref{lm-leray-seq}, with
$U=\Aa^1$, we must therefore bound $\dim E_2^{0,2}$. By the last part
of Proposition~\ref{pr-prop-seq} (2), we have
\begin{multline}\label{eq-ineq}
  \dim E_2^{0,2}\leq \dim E^2+\dim E_2^{2,1}= \dim
  H^2_c(\Aa^2\times\bFq,p_2^*\sheaf{F}\otimes \sheaf{K})+\\
  \dim H^2_c(\Aa^1\times\bFq,T^1_{\sheaf{K}}(\sheaf{F}))
  =h^2(f,\sheaf{F})+ \dim
  H^2_c(\Aa^1\times\bFq,T^1_{\sheaf{K}}(\sheaf{F})).
\end{multline}
\par
We have already recalled in the proof of Lemma~\ref{lm-cont-hic} that
$$
\dim H^2_c(\Aa^1\times\bFq,T^1_{\sheaf{K}}(\sheaf{F}))\leq
\rank(T^1_{\sheaf{K}}(\sheaf{F})).
$$
\par
Using the notation of Definition~\ref{def-spec}
%%, Lemma~\ref{lm-as},(2) 
and the proper base change theorem, we get
$$
\dim H^2_c(\Aa^1\times\bFq,T^1_{\sheaf{K}}(\sheaf{F})) \leq
\max_{x\in\Aa^1(\bFq)} \dim T^1_{\sheaf{K}}(\sheaf{F})_x \leq
\max_{x\in\bFq} \dim
H^1_c(\Aa^1\times\bFq,\sheaf{F}\otimes\sheaf{L}_x),
$$
and by Corollary~\ref{cor-fiber-dim}, this shows that
$(f,\sheaf{F})\mapsto \dim
H^2_c(\Aa^1\times\bFq,T^1_{\sheaf{K}}(\sheaf{F}))$ is continuous. The
inequality~(\ref{eq-ineq}) then finishes the proof of \eqref{tr}.
\par
The proof of \eqref{trb} is identical: it suffices to fix $\mcF=\bQl$
in the above argument.

\section{Proposition~\ref{pr-steps}: proof of \eqref{qu}}

We assume that
the function $f\mapsto c_i(f,\bQl)$ are continuous for $i=1$ and
$i=2$, and aim at proving that $(f,\mcF)\mapsto h^2(f,\sheaf{F})$ is continuous.  
\par
We apply the second spectral sequence of Lemma~\ref{lm-leray-seq},
with $V=\Aa^1$, and the first part of Proposition~\ref{pr-prop-seq}
(2) with $n=2$: this gives
$$
h^2(f,\sheaf{F})= \dim E^2\leq \dim E_2^{2,0}+\dim E_2^{1,1}+\dim
E_2^{0,2},
$$
where
$$
E_2^{i,j}=H^i_c(\Aa^1\times\bFq,\sheaf{F}\otimes T^j_{\sheaf{K}^*}
(\bQl)).
$$
\par
We note that $\cond(\sheaf{K}^*)=\cond(\sheaf{K})$. We have
$E_2^{2,0}=0$ (cf. (1) of Section \ref{sec-begin}), and
$$
\dim E_2^{1,1}=\dim H^1_c(\Aa^1\times\bFq,\sheaf{F}\otimes
T^1_{\sheaf{K}^*}(\bar{\Qq}_{\ell}))
$$
is continuous by Lemma~\ref{lm-cont-hic} and~\ref{lm-tensor-cond},
since the conductor of $T^1_{\sheaf{K}^*}(\bQl)$ is bounded
polynomially in terms of the conductor of $f$ by assumption.
\par
Finally, we have
$$
\dim E_2^{0,2}=\dim H^0_c(\Aa^1\times\bFq,\sheaf{F}\otimes
T^2_{\sheaf{K}^*}(\bQl))\leq \cond(\sheaf{F}\otimes
T^2_{\sheaf{K}^*}(\bQl)).
$$
\par
By assumption, $f\mapsto c_2(f^*,\bQl)$ is continuous, and therefore
the function $\dim E_2^{0,2}$ is continuous
(Lemma~\ref{lm-tensor-cond}). Thus $h^2(f,\sheaf{F})$ is also
continuous.

\section{Proposition~\ref{pr-steps}: proof of \eqref{ci} and \eqref{cib}}

We assume that the function $(f,\sheaf{F})\mapsto c_2(f,\sheaf{F})$ is
continuous, and  aim at proving that $(f,\mcF)\mapsto m(f,\sheaf{F})$ is continuous. We still assume that
$f\notin \Fq(X)+\Fq(Y)$. 
\par
We set
$$
\sheaf{G}=T^1_{\sheaf{K}}(\sheaf{F})=
R^1p_{1,!}(p_2^*\sheaf{F}\otimes\sheaf{K}). 
$$
We have to bound the rank $\rank(\mcG)$ and the punctual part $\pct(\mcG)$.
%, and finally the
%number of singularities $n(\mcG)$. 
\subsection{Bounding $\rank(\mcG)$}
For $x\in\Aa^1(\bFq)$, the stalk of $\sheaf{G}$ at $x$ is 
$$
\sheaf{G}_x=H^1_c(\Aa^1\times\bFq,\sheaf{F}\otimes\sheaf{L}_x)
$$
by the proper base change theorem. 
The generic rank of $\sheaf{G}$ is at most the maximal value of the
dimension of this stalk as $x$ varies. Hence, by Corollary~\ref{cor-fiber-dim}, it
is a continuous function of $(f,\sheaf{F})$.
\subsection{Bounding $\pct(\mcG)$}
We have
$$
\pct(\sheaf{G})=\dim H^0_c(\Aa^1\times\bFq,\sheaf{G}).
$$
\par
Again by (1) of Section \ref{sec-begin}, we have $T^0_{\mcK}(\mcF)=0$, therefore in the first spectral
sequence of Lemma~\ref{lm-leray-seq} (with $U=\Aa^1$) we have $E_2^{\ \!\! p,0}=0$ so that applying Proposition~\ref{pr-prop-seq} (3), we obtain 
$$
\pct(\sheaf{G})=\dim
H^1_c(\Aa^2\times\bFq,p_2^*\sheaf{F}\otimes\sheaf{K}).
$$

To bound this last quantity, we need the following cohomological lemma: 

\begin{lemma}\label{lm-ugly-1}
Let $\Fq$ be a finite field of characteristic $p$, $\ell\not=p$ a
prime number and $\psi$ a non-trivial $\ell$-adic additive character. Let $f=g_1/g_2\in\Fq(X,Y)$ be a rational function with $g_1,g_2\in\Fq[X,Y]$ coprime and 
$$
\sheaf{K}=\sheaf{L}_{\psi(f)}.
$$
Let
$\sheaf{F}$ be a constructible $\ell$-adic sheaf on $\Aa^1_{\Fq}$.

\par
Let $C$ be the union of the zero set of $g_2$, seen as a
reduced subscheme of $\Aa^2$, and of the lines
$$
\Aa^1\times\{y\}\subset \Aa^2,
$$
where $y\in\Aa^1(\bFq)$ is a singularity of $\sheaf{F}$. Let $W\subset
\Aa^2$ be the open subset complement of $C$. 
\begin{enumerate}
\item \label{ugly1} We have
$$
H^1_c({W}\times\bFq,p_2^*\sheaf{F}\otimes\sheaf{K})=0.
$$
\item \label{ugly2} The map
$$
(f,\sheaf{F})\mapsto \dim
H^1_c({C}\times\bFq,p_2^*\sheaf{F}\otimes\sheaf{K})
$$
is continuous. 
 	
\end{enumerate}
 
\end{lemma}
We give the proof of this lemma below, but let us explain first how to conclude the proof of \eqref{ci}.
\par 
Let $W\subset \Aa^2\times\Fq$ be the (dense) open set defined
in Lemma~\ref{lm-ugly-1} and $C=\Aa^2-W$ its complement. By
\eqref{ugly1} of this Lemma, we have 
$$H^1_c({W}\times\bFq,p_2^*\sheaf{F}\otimes\sheaf{K})=0$$
and from the excision inequality~(\ref{eq-sigma}), we get
$$
\pct(\sheaf{G})= \dim
H^1_c(\Aa^2\times\bFq,p_2^*\sheaf{F}\otimes\sheaf{K}) \leq \dim
H^1_c({C}\times\bFq,p_2^*\sheaf{F}\otimes\sheaf{K}),
$$
and the second part of Lemma~\ref{lm-ugly-1} shows that
$(f,\sheaf{F})\mapsto \pct(\sheaf{G})$ is continuous.

The proof of \eqref{cib} is identical: it suffices to fix $\mcF=\bQl$
in the above argument.

\begin{proof}[Proof of Lemma \ref{lm-ugly-1}]
  (1) The open subset $W$ is a smooth affine surface, and
  $p_2^*\sheaf{F}\otimes\sheaf{K}$ is lisse on $W$,
  so~(\ref{eq-affine-vanishing}) gives the vanishing of the first
  cohomology group.
\par
(2) Write $C_1$ for the zero set of $g_2$ (as a reduced scheme) and 
$$
C_2=\bigcup_{y\in \tilde{S}}\ \Aa^1\times\{y\},
$$
where $y$ ranges over those singularities of $\sheaf{F}$ in
$\Aa^1(\bFq)$ such that $\Aa^1\times \{y\}$ is \emph{not} contained in
$C_1$.
\par
Let $S=C_1\cap C_2$ be the intersection of these two sets; because of
the last restriction, this is a finite set, and its order is bounded
polynomially in terms of $\cond(\sheaf{F})$ and $\cond(f)$ (e.g., by
Bezout's Theorem for plane curves). Applying the excision exact
sequence~(\ref{eq-excision}) to $C$ and the complement $U$ (in $C$) of
the closed set $C_1$, we get by~(\ref{eq-sigma}) the bound
% an exact sequence
% $$
% \cdots
% %% \lra
% %% H^0_c(\bar{C}_1,p_2^*\sheaf{F}\otimes\sheaf{K})
% \lra H^1_c(\bar{U},p_2^*\sheaf{F}\otimes\sheaf{K})
% \lra H^1_c(\bar{C},p_2^*\sheaf{F}\otimes\sheaf{K})
% \lra H^1_c(\bar{C}_1,p_2^*\sheaf{F}\otimes\sheaf{K})
% \lra \cdots.
% $$
% \par
% In particular, we obtain 
% \par
% But $\sheaf{K}$ is zero on $C_1$, and therefore we obtain
$$
\dim H^1_c(\bar{C},p_2^*\sheaf{F}\otimes\sheaf{K}) \leq
\dim H^1_c(\bar{U},p_2^*\sheaf{F}\otimes\sheaf{K})+
\dim H^1_c(\bar{C}_1,p_2^*\sheaf{F}\otimes\sheaf{K})=
\dim H^1_c(\bar{U},p_2^*\sheaf{F}\otimes\sheaf{K})
$$
since $\sheaf{K}$, by definition, is zero on $C_1$. 
% \par
% Let $V\subset C_1$ be the subset where $\sheaf{K}$ is not lisse: this
% is an open dense subset, with finite complement $T$, which is
% contained in the intersection of the zero sets of $g_1$ and
% $g_2$. Applying excision one more time, we have an exact sequence
% $$
% \cdots \lra H^1_c(\bar{V},p_2^*\sheaf{F}\otimes\sheaf{K}) \lra
% H^1_c(\bar{C}_1,p_2^*\sheaf{F}\otimes\sheaf{K}) \lra
% H^1_c(\bar{T},p_2^*\sheaf{F}\otimes\sheaf{K}) \lra \cdots,
% $$
% but $\sheaf{K}$ is zero on $V$, while the finiteness of $T$ implies
% that $H^1_c(\bar{T},p_2^*\sheaf{F}\otimes\sheaf{K})=0$. Thus
% $H^1_c(\bar{C}_1,p_2^*\sheaf{F}\otimes\sheaf{K})=0$.
\par
% It suffices now to estimate the dimension of
% $H^1_c(\bar{U},p_2^*\sheaf{F}\otimes\sheaf{K})$. 
We have $U=C_2-S$, and we apply again the excision exact sequence to
$C_2$ and its open set $U$, obtaining by~(\ref{eq-sigma1}) the bound
% to get an exact sequence
% $$
% H^0_c(\bar{S},p_2^*\sheaf{F}\otimes\sheaf{K}) \lra
% H^1_c(\bar{U},p_2^*\sheaf{F}\otimes\sheaf{K}) \lra
% H^1_c(\bar{C}_2,p_2^*\sheaf{F}\otimes\sheaf{K})\lra \cdots
% %%\lra H^1_c(\bar{S},p_2^*\sheaf{F}\otimes\sheaf{K})=0
% $$
% and it follows that
$$
  \dim H^1_c(\bar{C},p_2^*\sheaf{F}\otimes\sheaf{K})\leq \dim
  H^1_c(\bar{U},p_2^*\sheaf{F}\otimes\sheaf{K})
  \leq \dim H^1_c(\bar{C}_2,p_2^*\sheaf{F}\otimes\sheaf{K})
$$
(because $H^0_c(\bar{S},p_2^*\sheaf{F}\otimes\sheaf{K})=0$ since
$S\subset C_1$, so that $\sheaf{K}$ is zero on $S$). Finally, we have
$$
H^1_c(\bar{C}_2,p_2^*\sheaf{F}\otimes\sheaf{K})=
\bigoplus_{y\in\tilde{S}} \sheaf{F}_y\otimes
H^1_c(\Aa^1\times\bFq,\sheaf{L}_{\psi(f(x,Y))}),
$$
and moreover both $\tilde{S}$ and $S$ have order bounded in terms of
$\cond(\sheaf{F})$ and $\cond(f)$, so that we obtain the result.
\end{proof}

\section{Proposition~\ref{pr-steps}: proof of \eqref{si} and \eqref{sib}}

We
assume that the functions $(f,\sheaf{F})\mapsto m(f,\sheaf{F})$ and
$(f,\sheaf{F})\mapsto h^2(f,\sheaf{F})$ are continuous, and aim at proving that
 $$(f,\sheaf{F})\mapsto c_1(f,\sheaf{F})$$ is continuous. We recall that
$$
\sheaf{G}=T^1_{\sheaf{K}}(\sheaf{F})=
R^1p_{1,!}(p_2^*\sheaf{F}\otimes\sheaf{K}). 
$$

\par
We have by definition
$$\cond(\mcG)=\cond(\mcG_0)+\pct(\mcG)$$
where $\mcG_0$ denote the middle-extension part of $\mcG$.  From the
assumption that $m(f,\sheaf{F})$ is continuous, the punctual part
$\pct(\mcG)$ and the $\rank(\mcG_0)$ are continuous; by
Lemma~\ref{lm-cd-bound} (applied to $\mcG_0$), it is enough to prove
that
$$
(f,\mcF)\mapsto \dim H^1_c(\Aa^1\times\bFq,T^1_{\sheaf{K}}(\sheaf{F}))
$$
is continuous.
\par
We use the first spectral sequence of Lemma~\ref{lm-leray-seq} (with
the open set $\Aa^1$).  By Proposition~\ref{pr-prop-seq} (1)
and~(\ref{eq-convergence}), we have
$$
E^2\simeq E_3^{0,2}\oplus E_2^{1,1}\oplus E_3^{2,0}
$$
and in particular
$$
\dim H^1_c(\Aa^1\times\bFq,T^1_{\sheaf{K}}(\sheaf{F}))= \dim E_2^{1,1}\leq
\dim E^2=\dim H^2_c(\Aa^2\times\bFq,p_2^*\sheaf{F}\otimes
\sheaf{K}).
$$
Since $h^2(f,\sheaf{F})$ is also
assumed to be continuous, this proves \eqref{si}. 

The proof of \eqref{sib} is identical: it suffices to fix $\mcF=\bQl$
in the above argument.

\section{Proposition~\ref{pr-steps}: proof of \eqref{se}}\label{sec-7}

We assume
that the functions $(f,\sheaf{F})\mapsto c_i(f,\sheaf{F})$ are
continuous for $0\leq i\leq 2$, and aim at proving
thaat $$(f,\sheaf{F})\mapsto h^j(f,\sheaf{F})$$ is continuous for
$j\leq 4$. 
\par
We use the first spectral sequence of Lemma~\ref{lm-leray-seq} with
$U=\Aa^1$. For any $j$, it implies that
$$
h^j(f,\sheaf{F})=\dim E^j\leq \sum_{p=0}^j
\dim H^p_c(\Aa^1\times\bFq,T^{j-p}_{\sheaf{K}}(\sheaf{F})).
$$
\par
By Lemma~\ref{lm-cont-hic}, and the continuity of
$c_{j-p}(f,\sheaf{F})$, each term in the sum is a continuous function,
and hence so is $h^j(f,\sheaf{F})$.

\section{Two special examples}\label{sec-polymath}

This section is largely independent of the full proof of
Proposition~\ref{pr-steps}. We establish Theorem~\ref{th-conductor} in
the special but fundamental case of the Fourier transform, and in a
related case which arose during the discussions related to the
\textsc{Polymath8} project \cite{polymath8ANT}.
\par
The complications which account for the length of the proof
Proposition~\ref{pr-steps}, compared with the case of the Fourier
transform, are that the cohomology of the specializations
$\sheaf{L}_{\psi(f(x,Y))}$ are not as simple as that of
$\sheaf{L}_{\psi(xY)}$ (for instance, it is not the case in general
that $\pct(\sheaf{L}_{\psi(f(x,Y))})=0$, as happens in the case of the
Fourier transform, see below).

\begin{remark} We will not strictly keep track of the
fact that the conductor bounds for the Fourier transform are of
polynomial size in terms of $\cond(\sheaf{F})$, but this is easily checked to follow from the argument.
\end{remark}
\subsection{The Fourier transform}
We consider the case 
$$
f(X,Y)=XY\in\Fq[X,Y],
$$
and we write $\ft_{\psi}(\sheaf{F})$ for the corresponding transform
$$
\ft_{\psi}(\sheaf{F})=R^1p_{1,!}(p_2^*\sheaf{F}\otimes\sheaf{L}_{\psi(XY)})(1/2),
$$
which is up to the Tate twist the ``naive'' Fourier transform of~\cite[Chap. 8]{katz-gkm}. Note that $\cond(f)=2$, independently
of $q$. We will not need, however to restrict to primes $p>2$.
\par
Let $\sheaf{F}$ be a middle-extension sheaf and $$\sheaf{G}=\ft_{\psi}(\sheaf{F}).$$ 
By Lemma~\ref{lm-cd-bound} it suffice to show that  $\rank(\mcG)$, $\pct(\mcG)$ and $h_1(\mcG)$ are bounded in terms of $\cond(\mcF)$.

 We start with the rank: by the proper base change
theorem, the fiber of $\mcG$ at $x\in\Aa^1(\bFq)$ is
$$
H^1_c(\Aa^1\times\bFq,\sheaf{F}\otimes\sheaf{L}_{\psi(xY)}).
$$
From Lemmas~\ref{lm-tensor-cond} and~\ref{lm-cont-hic}, we already see
that the maximum over $x$ of the dimension of these spaces, hence also
$\rank(\mcG)$, is bounded in terms of $\cond(\sheaf{F})$.
\par
We next claim that $\pct(\sheaf{G})=0$.
%; it follows then that
%$\sheaf{G}$ is lisse on $U_1$ (this is similar to, but slightly more
%precise than Lemma~\ref{lm-criterion-constant}), and consequently
%that $m(\sheaf{G})=\rank(\mcG)$ is bounded in terms of $\cond(\sheaf{F})$.
For this, we use the first spectral sequence of Lemma~\ref{lm-leray-seq} (taking
 $U=\Aa^1$ there) and apply Proposition~\ref{pr-prop-seq}, (3) to
deduce
$$
\pct(\sheaf{G})= \dim H^0_c(\Aa^1\times\bFq,\sheaf{G})=\dim
H^1_c(\Aa^2\times\bFq,p_2^*\sheaf{F}\otimes
\sheaf{L}_{\psi(XY)}).
$$
\par
Let $S\subset \Aa^1$ be the finite set of singularities of $\sheaf{F}$
in $\Aa^1$ and $T=\Aa^1\times S\subset \Aa^2$. The sheaf
$\sheaf{M}=p_2^*\sheaf{F}\otimes \sheaf{L}_{\psi(XY)}$ is lisse on the
dense open set $W=\Aa^2-T$. Applying excision~(\ref{eq-excision}), we
get an exact sequence
$$
\cdots \lra H^1_c(\bar{W},\sheaf{M})
\lra H^1_c(\Aa^2\times\bFq,\sheaf{M})
\lra H^1_c(\bar{T},\sheaf{M})\lra \cdots.
$$
\par
We have
$$
H^1_c(\bar{W},\sheaf{M})=0
$$
by~(\ref{eq-affine-vanishing}), because $W$ is an affine surface and
$\sheaf{M}$ is lisse on $W$. Also, since $T$ is a disjoint union of
``horizontal'' lines, we have
$$
H^1_c(\bar{T},\sheaf{M})=\bigoplus_{y\in S}H^1_c(\Aa^1\times\bFq,
\sheaf{L}_{\psi(yX)})\otimes\sheaf{F}_y=0
$$
because $H^1_c(\Aa^1\times\bFq,\sheaf{L}_{\psi(yX)})=0$ for all $y\in
S$ (including $y=0$). The excision exact sequence then gives
$H^1_c(\Aa^2\times\bFq,\sheaf{M})=0$, as claimed.
\par
By Lemma~\ref{lm-cd-bound}, we deduce that the conductor of
$\mcG$ is bounded in terms of the conductor of
$\sheaf{F}$, and of the invariant
$$
h_1(\sheaf{G})=\dim H^1_c(\Aa^1\times\bFq,\sheaf{G}).
$$
\par
By the first spectral sequence of Lemma~\ref{lm-leray-seq} 
 and Proposition~\ref{pr-prop-seq}, (1), we get
$$
\dim H^1_c(\Aa^1\times\bFq,\sheaf{G}) \leq \dim
H^2_c(\Aa^2\times\bFq,\sheaf{M}).
$$
\par
To compute this last group, we first use the second spectral sequence,
which shows that
$$
H^2_c(\Aa^2\times\bFq,\sheaf{M})\simeq
E_3^{0,2}\oplus E_3^{1,1}\oplus E_3^{2,0}\simeq E_3^{0,2}\oplus E_3^{1,1}.
$$
\par
We have
\begin{align*}
  E_2^{1,1}&=H^1_c(\Aa^1\times\bFq,R^1p_{2,!}(p_2^*\sheaf{F}\otimes
  \sheaf{L}_{\psi(XY)}))\\
  &=H^1_c(\Aa^1\times\bFq,\sheaf{F}\otimes R^1p_{2,!}\sheaf{L}_{\psi(XY)})
\end{align*}
by the projection formula. But the sheaf
$R^1p_{2,!}\sheaf{L}_{\psi(XY)}$ is zero, since the fiber at any $y\in
\Aa^1(\bFq)$ is
$$
H^1_c(\Aa^1\times\bFq, \sheaf{L}_{\psi(yX)})=0.
$$
\par
As for $E_3^{0,2}$, it is a subspace of
$$
E_2^{0,2}=H^0_c(\Aa^1\times\bFq,
R^2p_{2,!}(p_2^*\sheaf{F}\otimes\sheaf{L}_{\psi(XY)}))=
H^0_c(\Aa^1\times\bFq,\sheaf{F}\otimes R^2p_{2,!}\sheaf{L}_{\psi(XY)}).
$$
\par
The stalk of $R^2p_{2,!}\sheaf{L}_{\psi(XY)}$ at $y\in\bFq$ is
$$
H^2_c(\Aa^1\times\bFq,\sheaf{L}_{\psi(yX)})=
\begin{cases}
\bQl&\text{ if } y=0\\
0&\text{ otherwise,}
\end{cases}
$$
so the sheaf $R^2p_{2,!}\sheaf{L}_{\psi(XY)}$ is punctual and
supported at $0$ with stalk $\sheaf{F}_0$. Hence the dimension of
$E_2^{0,2}$ is at most the rank $\rank(\sheaf{F})\leq
\cond(\sheaf{F})$. Thus we obtain
$$
\dim H^2_c(\Aa^2\times\bFq,\sheaf{M})\leq \cond(\sheaf{F}).
$$
%\par
%We finally apply excision~(\ref{eq-excision}) again to compare
%$H^2_c(\Aa^2\times\bFq,\sheaf{M})$ and
%$H^2_c(\bar{U}_1\times\Aa^1,\sheaf{M})$. Let $C=\Aa^1-U_1$. We have an
%exact sequence
%$$
%\cdots\lra H^1_c(\bar{C},\sheaf{M})\lra
%H^2_c(\bar{U}_1\times\Aa^1,\sheaf{M})
%\lra H^2_c(\Aa^2\times\bFq,\sheaf{M})\lra \cdots
%$$
%so that (by the above) we have
%$$
%\dim H^2_c(\bar{U}_1\times\Aa^1,\sheaf{M})\leq \dim
%H^1_c(\bar{C},\sheaf{M}) +\cond(\sheaf{F})
%$$
%(we just spelled-out the inequality~(\ref{eq-sigma1}) explicitly.)
%\par
%From the definition of $C$, we get
%$$
%H^1_c(\bar{C},\sheaf{M})\simeq \bigoplus_{x\in(\Aa^1-U_1)(\bFq)}
%H^1_c(\Aa^1\times\bFq,\sheaf{F}\otimes\sheaf{L}_{\psi(xY)}).
%$$
%\par
%From Corollary~\ref{cor-fiber-dim}, we know that the maximum dimension
%of the
%$$
%H^1_c(\Aa^1\times\bFq,\sheaf{F}\otimes\sheaf{L}_{\psi(xY)})
%$$
%is bounded in terms of $\cond(\sheaf{F})$, and since we saw already
%that $|\Aa^1-U_1|$ is also bounded, we conclude that
%$\sigma(\ft_{\psi}(\sheaf{F}))$ is bounded in terms of the conductor
%of $\sheaf{F}$. 
This concludes the proof of Theorem~\ref{th-conductor}
for the Fourier transform. We state it formally for convenience:

\begin{corollary}
  Let $\Fq$ be a finite field of characteristic $p$, $\ell\not=p$ a
  prime number. Let $\psi$ be a non-trivial additive $\ell$-adic
  character of $\Fq$. There exists a function $n\mapsto C(n)$ with
  positive integral values such that, for any middle-extension sheaf
  $\sheaf{F}$ on $\Aa^1_{\Fq}$, the naive Fourier transform
  $\ft_{\psi}(\sheaf{F})$ satisfies $\pct(\ft_{\psi}(\sheaf{F}))=0$
  and we have
$$
\cond(\ft_{\psi}(\sheaf{F}))\leq C(\cond(\sheaf{F})).
$$
\end{corollary}

As we already mentioned in the introduction, we obtain
in~\cite[Prop. 8.2]{FKM1} the estimate
$$
\cond(\sheaf{\ft}_{\psi}(\sheaf{F}))\leq 10\cond(\sheaf{F})^2
$$
for $\sheaf{F}$ a Fourier sheaf on $\Aa^1_{\Fp}$, using the local
study of the Fourier transform, due to Laumon~\cite{laumon}. It is
clear that the arguments above  can also be used to give a completely effective upper bound.

\begin{remark}
(1)  A Fourier sheaf is defined to be a middle-extension sheaf which has
  no subsheaf or quotient sheaf geometrically isomorphic to an
  Artin-Schreier sheaf $\sheaf{L}_{\psi(aX)}$. For a sheaf which is
  not of this type, the naive Fourier transform is not the right
  object to consider, but this is of course not due to a failure of
  continuity.
\par
For instance, if $\sheaf{F}=\sheaf{L}_{\psi(Y)}$ (a typical
non-Fourier sheaf!) we have
$$
R^1p_{1,!}(p_2^*\sheaf{F}\otimes \sheaf{L}_{\psi(XY)})=0
$$
since the stalk of this sheaf at $x\in \Aa^1(\bFq)$ is
$$
H^1_c(\Aa^1\otimes\bFq,\sheaf{L}_{\psi((1+x)Y)})=0
$$
for all values of $x$. This certainly has bounded conductor!
\par
(2) For Fourier sheaves, other properties of the Fourier transform are
established, relatively elementarily, in~\cite[8.2.5,
8.4.1]{katz-gkm}: the Fourier transform is again a Fourier sheaf, and
the Fourier transform of a geometrically irreducible Fourier sheaf is
again geometrically irreducible.
\par
We can deduce a version of the irreducibility property (which suffices
in many applications) from the diophantine irreducibility criterion of
Lemma~\ref{lm-irred-crit}.  Indeed, if $\sheaf{F}$ is a
middle-extension Fourier sheaves which is pointwise pure of weight
$0$, we have the discrete Plancherel formula
$$
\frac{1}{q^{\nu}} \sum_{x\in
  \Fqn}|\frfn{\sheaf{F}}(x,q^{\nu})|^2
=
\frac{1}{q^{\nu}} \sum_{t\in
  \Fqn}|\frfn{\ft_{\psi}(\sheaf{F})}(t,q^{\nu})|^2
$$
for $\nu\geq 1$. The Fourier transform $\ft_{\psi}(\sheaf{F})$ is
mixed of weight $\leq 0$ by the Riemann Hypothesis (in fact, it is
known to be pure of weight $0$, but this is again a deeper fact),
hence Lemma~\ref{lm-irred-crit} implies that $\sheaf{F}$ is
geometrically irreducible if and only if the weight $0$ part of
$\ft_{\psi}(\sheaf{F})$ is geometrically irreducible.
\end{remark}
\subsection{The \textsc{Polymath8} kernel}
We next consider  another example discussed in the blog of the \textsc{Polymath8}
project, which we will reduce to a Fourier transform. We let
$$
f=\frac{1}{X(X+Y)}+hY
$$
where $h\in\Fq^{\times}$ is a parameter, and we wish to bound the
conductor of
$$
R^1p_{1,!}\sheaf{L}_{\psi(f)}(1/2),
$$
i.e., the corresponding transform of the trivial sheaf, by a constant
(independent of $q$). 
% Note that we could also view this as the transform associated to
% $1/(X(X+Y))$, applied to
% $\sheaf{F}=\sheaf{L}\otimes\sheaf{L}_{\psi(hY)}$, but this is not
% particularly needed here.
\par
We outline the steps that prove such a bound, leaving some details to
the reader.
\par
-- It is equivalent to bound the conductor of
$$
R^1p_{1,!}\sheaf{L}_{\psi(g)}(1/2)
$$
where $g=(XY)^{-1}+hY-hX$ (applying the automorphism $(X,Y)\mapsto
(X,X+Y)$). By the projection formula, we have 
$$
R^1p_{1,!}\sheaf{L}_{\psi(g)}(1/2)=\sheaf{L}_{\psi(-hX)}\otimes\sheaf{G},
$$
where
$$
\sheaf{G}=R^1p_{1,!}\sheaf{L}_{\psi(h)}(1/2),\quad\quad h=\frac{1}{XY}+hY.
$$
\par
By Lemma~\ref{lm-tensor-cond}, it is enough to estimate the conductor
of $\sheaf{G}$. 
\par
-- Note that the trace function of $\sheaf{G}$ is
$$
\frac{1}{q^{1/2}}\sum_{y\not=0}\psi\Bigl(\frac{1}{xy}+hy\Bigr)=\frac{1}{q^{1/2}}\sum_{v\not=0}
\psi\Bigl(\frac{1}{v}+\frac{h}{x}v\Bigr)
$$
for $x\not=0$, which is visibly a normalized Kloosterman sum with parameter
$h/x$. Let
$$
\pi\,:\,\begin{cases}
  \Gg_m\times\Gg_m\lra \Gg_m\times\Gg_m\\
  (x,y)\mapsto (hx^{-1},xy)
\end{cases}
$$
and 
$$
\nu\,:\, \begin{cases}
  \Gg_m\lra \Gg_m\\
  x\mapsto hx^{-1}
\end{cases}
$$
\par 
Then the analogue of the change of variable $(u,v)=\pi(x,y)=(h/x,xy)$
that establishes this identity is the isomorphism
$$
\nu^*R^1p_{1,!}\sheaf{L}_{\psi(V^{-1}+UV)}\simeq
R^1p_{1,!}\sheaf{L}_{\psi((XY)^{-1}+hY)}
$$
of sheaves over the multiplicative group $\Gg_m=\Aa^1-\{0\}$ over
$\Fq$, which is a consequence of the isomorphism
$\pi^*\sheaf{L}_{\psi(V^{-1}+UV)}\simeq
\sheaf{L}_{\psi((XY)^{-1}+Y)}$.
\par
Note that
$$
R^1p_{1,!}\sheaf{L}_{\psi(V^{-1}+UV)}(1/2)=\ft_{\psi}(\sheaf{L}_{\psi(V^{-1})}),
$$
which has bounded conductor independently of $q$. Since $\nu$ is an
automorphism, and since the dimensions of the stalk of $\sheaf{G}$ at
$0$ is bounded, it follows from the fact that $\sheaf{G}$ coincides
with $(\nu^{-1})^*\ft_{\psi}(\sheaf{L}_{\psi(V^{-1})})$ on $\Gg_m$
that the conductor of $\sheaf{G}$ is bounded for all $q$, as desired.

\begin{remark}
The Fourier transform of $\sheaf{L}_{\psi(X^{-1})}$ is the
\emph{Kloosterman sheaf} (in one variable), that was originally
defined by Deligne. See~\cite{katz-gkm} for its properties, and
generalizations to more than one variable.
\end{remark}

\end{document}